\documentclass[11pt]{article}
\usepackage[utf8]{inputenc}
{\tiny {\tiny }}
\usepackage[numbers,sort&compress]{natbib}

\makeatletter

\@ifundefined{email}{
	
}{}
\makeatother

\usepackage{natbib}       
\usepackage{hyperref}
\usepackage{algorithm}
\usepackage{algorithmic}
\usepackage{comment}
\usepackage{pgf}
\usepackage{tikz}
\usetikzlibrary{arrows, decorations.pathmorphing, backgrounds, positioning, fit, petri, automata}
\definecolor{yellow1}{rgb}{1,0.8,0.2}

\usepackage{makecell}
\usepackage[numbers]{natbib} 
\usepackage{mathtools}

\usepackage{amsmath}
\usepackage{amsmath,amsthm,amsfonts,amssymb,amsbsy,amsopn,amstext}
\usepackage{graphicx,color}
\usepackage{mathbbol,mathrsfs}
\usepackage{float,ccaption}
\usepackage{indentfirst}
\usepackage{appendix}
\usepackage{subfigure}
\usepackage{graphics} 
\usepackage{epsfig} 
\usepackage{epstopdf}
\usepackage{subcaption}

\usepackage{emptypage}

\usepackage{booktabs}
\usepackage{threeparttable}
\usepackage{multirow}
\usepackage{diagbox}

\newtheorem{thm}{Theorem}

\newtheorem{lem}{Lemma}
\newtheorem{prop}{Proposition}

\newtheorem{ass}{Assumption}

\newcommand{\Var}{\ensuremath{\mathrm{Var}}} 

\newcommand{\normm}[1]{{\left\vert\kern-0.25ex\left\vert\kern-0.25ex\left\vert #1
		\right\vert\kern-0.25ex\right\vert\kern-0.25ex\right\vert}}

\usepackage{geometry}
\begin{document}

\title{Trust Region Algorithm for Stochastic Minimax Problems with Decision-Dependent Distributions}
\author{Yan Gao, Yongchao Liu\thanks{School of Mathematical Sciences, Dalian University of Technology, Dalian 116024, China, e-mail: gydllg123@mail.dlut.edu.cn (Yan~Gao), lyc@dlut.edu.cn (Yongchao~Liu), 979510015@mail.dlut.edu.cn (Zili~Luo).}, Zili Luo}

\date{}
\maketitle
\noindent{\bf Abstract.} 	
Stochastic minimax optimization has drawn much attention over the past decade due to its broad applications in machine learning, signal processing and game theory. 
In some applications, the probability distribution of uncertainty depends on decision variables, as the environment may respond to decisions. 
In this paper, we propose a trust region algorithm for finding the stationary points of stochastic minimax problems with decision-dependent distributions. At each iteration, the algorithm locally learns the dependence of the random variable on the decision variable from data samples via linear regression, and updates the decision variable by solving trust-region subproblem with the learned distribution. When the  objective function is nonconvex--strongly concave and the distribution map follows a regression model, we prove the almost sure convergence of the iterates to a stationary point of the primal function. The effectiveness of the proposed algorithm is further demonstrated through numerical experiments on both synthetic and real-world data sets.

\section{Introduction}	
Stochastic minimax optimization, which captures the nested structure by simultaneously minimizing and maximizing over two subsets of variables, arises in a variety of areas, including adversarial machine learning \cite{Creswell2018Generative, Goodfellow2014Advance}, distributionally robust optimization \cite{Levy2020Large,namkoong2016stochastic,zhu2023distributionally},  reinforcement learning \cite{sutton1988learning, wai2018multi}, game theory \cite{Carmon2019Variance,Nash1953Two}, among others. 
In some applications, the distributions of stochastic elements may depend on or shift in response to decision variables. 
For example, demand depends on price \cite{cheung2017dynamic,cooper2006models}, traffic predictions influence traffic patterns \cite{liebig2017dynamic, perdomo2020performative}, and predictions of credit default risk influence interest rate assignments and hence default rates \cite{inga2022credit,robinson2024loan}. 
When the decision variable in a stochastic minimax problem affects the probability distribution, the corresponding stochastic minimax problems with decision-dependent distributions (SMDD)~\cite{wood2023stochastic} can be formulated as follows:
\begin{equation}\label{eq:f_minimax_D}
	\min_{x\in\mathcal{X}}\max_{y\in \mathcal{Y}}\mathcal{L}(x,y)\coloneqq\underset{\omega\sim \mathcal{D}(x, y)}{\mathbb{E}}\left[l(x,y,\omega)\right],
\end{equation}
where $\omega$ is a random variable supported on $\mathcal{Z}\subset\mathbb{R}^{d}$, $\mathcal{X}\subset\mathbb{R}^{n}$, $\mathcal{Y}\subset\mathbb{R}^{m}$,  $l(\cdot):\mathbb{R}^{n}\times\mathbb{R}^{m}\times\mathbb{R}^{d}\rightarrow\mathbb{R}$ is a smooth function, $\mathcal{D}(\cdot):\mathbb{R}^{n}\times\mathbb{R}^{m}\rightarrow\mathcal{P}(\mathbb{R}^{d})$ is a distribution map, 
and $\mathbb{E}\left[\,\cdot\,\right]$ denotes the expectation with respect to the distribution $\mathcal{D}(x,y)$. 

Stochastic minimax problem  with decision-dependent distribution is proposed by Wood and Dall'Anese~\cite{wood2023stochastic}. 
Theoretically, the authors introduce the notion of performative equilibrium point which is the saddle point for the stochastic minimax problem whose probability distribution is induced by itself, and provide sufficient conditions that guarantee the existence and uniqueness of the  performative equilibrium point. Algorithmically, they
propose (stochastic) primal-dual algorithm for seeking the performative equilibrium point of SMDD. Under constant and dynamic stepsize policies, the authors show that the proposed algorithm achieves a linear convergence rate and an $\mathcal{O}(\frac{1}{\sqrt{t}})$ convergence rate,  respectively. Moreover, when $\mathcal{L}(\cdot)$ is strongly convex--strongly concave, a zeroth-order algorithm for finding the saddle point of SMDD is also proposed. 
Narang et al.~\cite{narang2023multiplayer} study a decision-dependent  stochastic game model and introduce the notion of performatively stable equilibrium, which is the Nash equilibrium of the game whose probability distribution is induced by itself. The authors show the existence and uniqueness of the performatively stable equilibrium
when the stochastic game is strongly monotone, and propose the repeated retraining and repeated (stochastic) gradient methods for finding the performatively stable equilibrium. When the decision-dependent stochastic game is strongly monotone, they propose an adaptive stochastic gradient method for finding the Nash equilibrium. Moreover, when the distribution map follows a location-scale model, the authors show that the proposed algorithm achieves a  convergence rate $\mathcal{O}(\frac{1}{\sqrt{t}})$.
In the setting of monotonicity of the decision-dependent stochastic games, Wood et al. \cite{wood2024solvingdecisiondependentgameslearning} propose a two-stage method for seeking the Nash equilibrium of the stochastic game, which estimates the distribution map first and then solves the game by stochastic gradient descent method. More recently, Gao and Liu~\cite{Gao2025Adaptive} propose the adaptive stochastic gradient descent ascent algorithm and its alternating variant to find the stationary points of SMDD, which learns the unknown distribution map dynamically and optimizes the minimax problem simultaneously via (alternating) stochastic gradient descent ascent at each iteration. When the distribution map follows a location-scale model, the authors analyze the non-asymptotic complexity under nonconvex--(strongly) concave and nonconvex--P{\L} settings of the  objective function, respectively.

The methods that learn the distribution map $\mathcal{D}(\cdot)$ from a prespecified parametric function class prior to or concurrently with optimization \cite{Gao2025Adaptive, narang2023multiplayer,wood2024solvingdecisiondependentgameslearning} may inadequately approximate the true distribution map. On the other hand, the zeroth-order method \cite{wood2023stochastic} that estimates the gradient from data samples is only applicable to SMDD with strongly convex--strongly concave  objective function.
In this paper, we consider a non-parametric function class, more specifically, we assume that the random variable $\omega$ follows a regression model with respect to the decision variable $x$, i.e.,
$$\omega=m(x,\tilde{\epsilon})\coloneqq\psi(x)+\tilde{\epsilon},$$  where 
$\psi:\mathbb{R}^{n}\rightarrow \mathbb{R}^{d} $ is an regression function, and
$\tilde{\epsilon}\in\mathbb{R}^{d}$ is a random variable independent of $x$ and $y$. 
If the regression function $\psi(\cdot)$ and the distribution of $\tilde{\epsilon}$ are known, SMDD \eqref{eq:f_minimax_D} is equivalent to the following standard stochastic minimax problem
\begin{equation}
	\label{eq:f_minimax_regression}
	\min_{x\in\mathcal{X}}\max_{y\in \mathcal{Y}}\mathcal{L}(x,y)\coloneqq\mathbb{E}_{\tilde{\epsilon}}\left[l(x, y, \psi(x) +\tilde{\epsilon})\right].
\end{equation}
Unfortunately, both the regression function $\psi(\cdot)$ and the random variable $\tilde{\epsilon}$ are typically unknown, and the realizations of $\tilde{\epsilon}$ are unobservable, whereas only the realizations of $\omega=\psi(x)+\tilde{\epsilon}$ can be observed for any $x \in \mathbb{R}^n$.
Motivated by \cite{liu2022coupled},
we propose a trust region (TR) algorithm that approximates the regression model $m(\cdot)$ in local regions using the data samples of $(x,\omega)$ via linear regression. At each iteration, we construct the trust region subproblem with the locally learned regression model $m^{k}(\cdot)$, and update the decision variable by taking one gradient step on the resulting trust-region subproblem.
Although the proposed algorithm is primarily inspired by \cite{liu2022coupled}, 
the minimax structure and the intractability of obtaining an exact solution to the inner maximization problem make it a nontrivial extension.
Moreover, the proposed TR algorithm is not applicable to the case where the underlying random variable depends on both the outer variable $x$ and the inner variable $y$. This is because the trust region approximation model for the primal function
$\Phi(x)\coloneqq\max_{y\in \mathcal Y} \mathcal{L}(x,y)$ becomes computationally infeasible in such case, as the distribution map is only approximated locally.
On the other hand, in contrast to \cite{narang2023multiplayer,wood2024solvingdecisiondependentgameslearning} that approximate the distribution map of the random variable globally via parametric model in an offline or online manner, our approach employs local linear regression within a sequence of trust regions, which enables a more accurate and flexible approximation of the true distribution map. 

Indeed, trust region methods have found widespread applications in solving nonconvex minimax problems.
Yao and Xu \cite{yao2024trustregiontypealgorithms} propose the Minimax Trust Region algorithms with fixed trust region radius and dynamic trust region radius for solving nonconvex--strongly concave minimax problems. Both algorithms find an $\mathcal{O}(\epsilon, \sqrt{\epsilon})$-second-order stationary point within $\mathcal{O}\left(\epsilon^{-3/2}\right)$ iterations.
Wang and Xu \cite{wang2025gradientnormregularizationsecondorder} propose a Gradient norm Regularized Trust Region (GRTR) algorithm, which finds an $(\epsilon, \sqrt{\epsilon})$-second-order stationary point within $\tilde{\mathcal{O}}\left(\epsilon^{-3/2}\right)$ iterations.
The authors further propose the Levenberg-Marquardt algorithm with Negative Curvature to reduce the computational burden of GRTR for solving trust region subproblems, and it achieves the same complexity as GRTR.
Qiu et al.~\cite{Qiu2024aquasiNewton} propose a Quasi-Newton Subspace Trust Region algorithm for nonconvex--nonconcave stochastic minimax problems with box constraints and prove the global convergence to the $\epsilon$-first-order stationary point. 
We direct the interested readers to \cite{Sobky2018Adaptive,Wang2012Anewtrust,WANG20138033,yang2020global} for new development of trust region algorithms on minimax problems.
Notably, the aforementioned trust region algorithms are second-order methods that approximate the  objective function using second-order information within each trust region. The proposed TR algorithm approximates the unknown distribution map of the stochastic  objective function with linear regression within each trust region, and the employment of TR algorithm only involves first-order information. 
Different from the deterministic algorithms \cite{yao2024trustregiontypealgorithms,{wang2025gradientnormregularizationsecondorder}}, we can only establish almost sure convergence of the TR algorithm, as the approximation of the primal function and its estimated value lead to inexact sufficient descent condition and step acceptance criterion.

The structure of this paper is organized as follows.
Section~\ref{section2} introduces the proposed TR algorithm and presents necessary assumptions on SMDD~\eqref{eq:f_minimax_regression}. Section \ref{section3} presents the convergence analysis of the TR algorithm. 
Section~\ref{section4} verifies the effectiveness of the TR algorithm on a simple synthetic example and a real-world application. 

\noindent\textbf{Notation.} $\mathbb{R}^{n}$ denotes the $n$-dimensional Euclidean space endowed with norm $\|x\|=\sqrt{\langle x,x\rangle}$. $\mathbb{N}$ denotes the set of nonnegative integers.
$\mathcal{B}(x,\delta)$ denotes the ball centered at $x$ with radius $\delta>0$, i.e., $\mathcal{B}(x,\delta) \coloneqq  \{ y \in \mathbb{R}^n : \|x-y\| \leq \delta \}.$
For a matrix $A\in\mathbb{R}^{n\times m}$, $\left\|A\right\|_F$ denotes the Frobenius norm and $A^\top$ denotes its transpose. 
For a differentiable function $f(\cdot),$ $\nabla f(\cdot)$ denotes its full gradient, $\nabla_{x}f(\cdot)$ and $\nabla_{y}f(\cdot)$ denote the partial gradients with respect to $x$ and $y$, respectively,
and \( \nabla_i f(x_1, x_2, x_3) \) denotes the partial gradient of \( f(\cdot) \) with respect to the \( i \)-th block variable \( x_i \), for \( i = 1, 2, 3 \).
We denote $a=\mathcal{O}(b)$ if $\exists~C>0$ such that $|a|\leq C|b|$.
Given a discrete set $\mathcal{S}$, we denote its cardinality by $|\mathcal{S}|$. 

\section{TR Algorithm }\label{section2}
Before presenting the algorithm, we introduce the local linear regression (LLR) model and some notation used in the TR algorithm.

Assuming that the independent and identically distributed (i.i.d.) samples from the probability distribution $\mathcal{D}(x)$ can be drawn for any given $x$, we approximate the unknown distribution map $\mathcal{D}(\cdot)$ in local regions using a linear regression model. 
Specifically, we consider a hypothesis class of affine functions
\[
\mathcal{H} \coloneqq \left\{ \phi(\cdot): \mathbb{R}^{n} \rightarrow \mathbb{R}^{d} \;\middle|\; \phi(x) = (B^1)^\top x + (B^0)^\top, \; B^1 \in \mathbb{R}^{n \times d}, \; B^0 \in \mathbb{R}^{1 \times d} \right\}.
\]
At the $k$-th iteration, given the current iterate $\widehat{x}^{k}$ and the corresponding trust-region radius $\delta_k$, we construct a data set $T_k = \{(x^i, \omega^i)\}_{i=1}^{N_k},$
where the points $\{x^i\}_{i=1}^{N_k}$ are drawn uniformly from the local region $\mathcal{B}(\widehat{x}^{k}, \delta_k)$ and $\omega^i \sim \mathcal{D}(x^i)$.
Among the parametric functions in the hypothesis class  $\mathcal{H}$, we seek the one that minimizes the sum of square errors on the data set $T_{k}$. 
Then the estimation of parameters $\widehat{B}^{k,1}, \widehat{B}^{k,0}$ and residuals $\{e^{k,i}\}$ are constructed as follows:
\begin{equation*}
	\begin{aligned}
		\left\{\widehat{B}^{k, 0}, \widehat{B}^{k, 1}\right\}\in \underset{B^{0}, B^{1}}{\operatorname{argmin}} \sum_{\left(x^{i}, \omega^{i}\right) \in T_{k}}\left\| \omega^{i}-\left(B^{1}\right)^{\top} x^{i}-\left(B^{0}\right)^{\top}\right\|^{2},
	\end{aligned}
\end{equation*}
\begin{equation*}
	\begin{aligned}
	e^{k, i} \coloneqq \omega^{i}-\left(\widehat{B}^{k, 1}\right)^{\top} x^{i}-\left(\widehat{B}^{k, 0}\right)^{\top},\quad\text {for}\ i=1,2,\cdots, N_{k}.
	\end{aligned}
\end{equation*}
Let $\widehat{\epsilon}^{k}$ denote a random variable with the empirical probability distribution of $\left\{e^{k, i}\right\}_{i=1}^{N_{k}}$. 
Then, the resulting LLR model $m_k(\cdot)$ is defined as
\begin{equation}
	\label{eq:LLR}
	m_{k}\left(\widehat{x}^{k}+s, \widehat{\epsilon}^{k}\right) \coloneqq\left(\widehat{B}^{k, 1}\right)^{\top}\left(\widehat{x}^{k}+s\right)+\left(\widehat{B}^{k, 0}\right)^{\top}+\widehat{\epsilon}^{k}.
\end{equation}

Given the data set  $S = \{ \omega^j \}_{j=1}^{|S|}$ with $\omega^j \stackrel{\text{i.i.d.}}{\sim} \mathcal{D}(x)$, we define the following notation used throughout the paper,
\begin{equation}\label{TR_notations}
	\begin{aligned}
		&y^{*}(x)\coloneqq \arg\max_{y\in \mathcal Y} \mathcal{L}(x,y),\,\,
		\Phi(\cdot)\coloneqq \max _{y \in \mathcal{Y}} \mathcal{L}(\cdot, y),  \\
		&\mathcal{L}^{k}(x,y)\coloneqq  \mathbb{E}_{\widehat{\epsilon}^k}\left[l(x,y,m_k(x,\widehat{\epsilon}^k))\right],\;\;
		y^{k,*}(x)\coloneqq \arg\max_{y\in \mathcal{Y}}\mathcal{L}^k{(x,y)},\\
		&y^{s,*}(x)\coloneqq \arg\max_{y\in\mathcal{Y}}\frac{1}{\vert S\vert}\sum_{\omega^j\in S}l(x, y, \omega^j),
	\end{aligned}
\end{equation}
where $\mathcal{L}^{k}(\cdot)$ is the local approximate of $\mathcal{L}(\cdot)$ at the trust region $\mathcal{B}(\widehat{x}^{k}, \delta_k)$. Moreover, we use $y^{i,*}(x)$ and $y^{is,*}(x)$ as the inexact substitutes of $y^{k,*}(x)$ and $y^{s,*}(x)$, respectively.

The proposed TR algorithm reads as follows.
\begin{algorithm}[H] 
	\caption{Trust Region Algorithm (TR)}
	\label{TR:algorithmCL}
	\begin{algorithmic}[1]
		\STATE \textbf{Initialization:}
		$\widehat{x}^{0}\in \mathbb{R}^n, \delta_{0}\in (0,\delta_{max})$ with $\delta_{max}>0,\gamma>1,\eta_1\in (0,1),\eta_2>0.$
		\FOR{$k=1$ \TO $T$}
		\STATE Generate a set of samples $\left\{u^{i}\right\}$ from a uniform distribution  $\mathcal{U}\left(\mathcal{B}\left(\mathbf{0}_{n}, 1\right)\right)$. Let the data set $T_{k} \coloneqq\left\{\left(x^{i}, \omega^{i}\right)\right\}$, with $\left\{x^{i}\right\} \coloneqq\left\{\widehat{x}^{k}+\delta_{k} u^{i}\right\}$ and $\left\{\omega^{i}\right\}$ are the corresponding scenarios following the true models  $\left\{m\left(x^{i}, \widetilde{\epsilon}\right)\right\}$.
		\STATE Construct a local linear regression model $m_{k}\left(\widehat{x}^{k}+s, \widehat{\epsilon}^{k}\right)$ by (\ref{eq:LLR}) with the data set $T_{k}$.
		\STATE Construct the trust region subproblem:
		\begin{equation}\label{TR_trustRegionMinimization}
			\min_{\substack{\|s\|\leq\delta_{k}}}\Phi^{k}(s)\coloneqq \max_{y\in \mathcal{Y}}\mathcal{L}^k{(\widehat{x}^k+s,y)}.
		\end{equation}
		\STATE Solving TR subproblem:
		compute a trial step $s^{k}$ satisfying
		{\begin{equation}
				\label{TR_algo_inex_suffdescent}
				\mathcal{L}^{k}(\widehat{x}^k,y^{i,*}(\widehat{x}^k))-\mathcal{L}^{k}(\widehat{x}^k+s^k,y^{i,*}(\widehat{x}^k+s^k)) \geq \kappa_{dcp} \|\nabla_{1} \mathcal{L}^{k}(\widehat{x}^{k},y^{i,*}(\widehat{x}^{k}))\| \min \left\{\delta_{k}, 1\right\}.
		\end{equation}}
		\STATE Compute $v_{k}$ via \eqref{TR_v_k} and $v_{k+1/2}$ via \eqref{TR_v_k_2}, and set
		\begin{equation}\label{equ:ratio}
			\rho_{k}\coloneqq
			\frac{v_{k}-v_{k+1/2}}
			{\mathcal{L}^{k}(\widehat{x}^k,y^{i,*}(\widehat{x}^k))
				-\mathcal{L}^{k}(\widehat{x}^k+s^k,y^{i,*}(\widehat{x}^k+s^k))}.
		\end{equation}
		\IF {$\rho_k\geq \eta_1$ and $\|\nabla_{1} \mathcal{L}^{k}(\widehat{x}^{k},y^{i,*}(\widehat{x}^{k}))\| \geq \eta_{2}\delta_k$}
			\STATE $\widehat{x}^{k+1}=\widehat{x}^{k}+s^{k}, \quad \delta_{k+1}=\min \left\{\gamma \delta_{k}, \delta_{\max} \right\},$
		\ELSE 
			\STATE $\widehat{x}^{k+1}=\widehat{x}^{k}, \quad \delta_{k+1}=\gamma^{-1} \delta_{k}.$
		\ENDIF
		\ENDFOR
	\end{algorithmic}
\end{algorithm}

Algorithm~\ref{TR:algorithmCL} generalizes the CLEO algorithm \cite{liu2022coupled} designed for stochastic optimization with decision-dependent distributions to stochastic minimax problems. 
Steps 3--4 of Algorithm~\ref{TR:algorithmCL} construct a local linear regression model $m_k(x, \widehat{\epsilon}^k)$ to approximate the true distribution map $m(x, \tilde{\epsilon})$ within the trust region $\mathcal{B}(\widehat{x}^k, \delta_k)$. By replacing the true distribution map with $m_k(x, \widehat{\epsilon}^k)$ in $\mathcal{B}(\widehat{x}^k, \delta_k)$, we construct the trust-region subproblem in Step 5, provided that $\mathcal{L}^{k}(\cdot)$ is strongly concave in $y$.
We will show in Section~\ref{section3} that the trust region model $\Phi^{k}(\cdot)$ is a probabilistically accurate approximation to the primal function $\Phi(\cdot)$ with a sufficiently large number of samples.
In Algorithm~\ref{TR:algorithmCL}, Step 6, we solve the trust-region subproblem \eqref{TR_trustRegionMinimization}. 
Note that $\Phi^{k}(\cdot)$ is nonconvex, we define an inequality to characterize its descent property. Motivated by \cite{liu2022coupled}, we update the decision variable to obtain a trial step satisfying the sufficient descent inequality \eqref{TR_algo_inex_suffdescent}. 
Step 8  of Algorithm~\ref{TR:algorithmCL} performs the acceptance test. Different from the aforementioned trust region methods \cite{wang2025gradientnormregularizationsecondorder,yao2024trustregiontypealgorithms,Qiu2024aquasiNewton} and CLEO in \cite{liu2022coupled},
we cannot evaluate either the primal function $\Phi(\cdot)$ or its approximation $\Phi^{k}(\cdot)$, as the true distribution map $m(\cdot)$ is unknown and the exact solution of the inner maximization problem may be unavailable. Therefore, we define the new ratio of actual-to-predicted reduction \eqref{equ:ratio} to assess the trial step $s^k$, where
\begin{equation}\label{TR_v_k}
	v_k \coloneqq \frac{1}{\vert S_k\vert}\sum_{\omega^j\in S_k}l(\widehat{x}^{k}, y^{is,*}(\widehat{x}^k), \omega^j),
\end{equation}
\begin{equation}\label{TR_v_k_2}
	v_{k+1/2} \coloneqq \frac{1}{\vert S_{k+1/2}\vert}\sum_{\omega^{j}\in S_{k+1/2}}l(\widehat{x}^{k}+s^{k}, {y}^{is,*}(\widehat{x}^{k}+s^k), \omega^{j}),
\end{equation}
are value estimates of $\Phi(\cdot)$ at $\widehat{x}^k$ and $\widehat{x}^k+s^k$, respectively, $S_k=\{\omega^i\}_{i=1}^{|S_k|}$ and $S_{k+1/2}=\{\omega^j\}_{j=1}^{|S_{k+1/2}|}$ are sample sets with $\omega^i\sim\mathcal{D}(\widehat{x}^k)$ and $\omega^j\sim\mathcal{D}(\widehat{x}^k+s^{k})$, respectively.
If the acceptance criterion holds, a new iterate is accepted and the trust-region radius is possibly increased. Otherwise the step is rejected and the trust-region radius is decreased.
Notably, in the following analysis, we characterize the accuracy of the inexact solutions in Algorithm \ref{TR:algorithmCL} as $\|y^{i,*}(\widehat{x}^{k})-y^{k,*}(\widehat{x}^{k})\|\leq\epsilon$ and $\|y^{is,*}(\widehat{x}^{k})-y^{s,*}(\widehat{x}^{k})\|\leq\epsilon$.

At the end of this section, we present some assumptions and technical lemmas.
\begin{ass}\label{TR_ass_primal}
	The primal function $\Phi(\cdot)\coloneqq \max_{y\in \mathcal Y} \mathcal{L}(\cdot,y)$ is bounded below.
\end{ass}

\begin{ass}\label{TR_ass_psi}
	The regression function $\psi(\cdot):\mathbb{R}^{n}\rightarrow \mathbb{R}^{d}$ is a twice-differentiable function, $\tilde{\epsilon}$ is a random vector with compact sample space. Moreover, $\tilde{\epsilon}$ is independent of $x$ and $y$ and has zero mean and finite variance $\Sigma$.
\end{ass}

\begin{ass}\label{Ass:smoothness_local}
	$l(\cdot)$ is $\ell_{1}$-smooth and $L_1$-Lipschitz continuous,  and $\psi(\cdot)$ is  $\ell_{0}$-smooth and  $L_0$-Lipschitz continuous. Moreover, $\mathcal{X}=\mathbb{R}^n$, $\mathcal{Y}$ is a convex and bounded set with diameter $D>0$.
\end{ass}

\begin{ass}
	\label{ass:stronglyconcave_local}
	$\ell(x,\cdot,z)$ is $\mu$-strongly concave over $\mathcal{Y}$ for any $(x,z)\in\mathbb{R}^{n}\times\mathbb{R}^{d}$.
\end{ass}
Assumption \ref{TR_ass_primal} guarantees the feasibility of SMDD \eqref{eq:f_minimax_regression}. Assumption \ref{TR_ass_psi} characterizes the properties of the regression model of the random variable. Assumptions \ref{Ass:smoothness_local}-\ref{ass:stronglyconcave_local} guarantee the uniqueness of the optimal solution for the maximization problem $\max_{y\in\mathcal{Y}}\mathcal{L}(x,y)$ and ensure that the primal function $\Phi(\cdot)$ is differentiable.

\begin{lem}\label{TR_Lemma1}
	Suppose that Assumption \ref{Ass:smoothness_local} holds and $\|\widehat{B}^{k,1}\|_F\leq B$ for $\forall k\in \mathbb{N}$. Then $\mathcal{L}^{k}(\cdot)$ is $\widehat{L}_1$-Lipschitz continuous and $\widehat{\ell_1}$-smooth with $\widehat{L}_1=L_{1}\left(1+B\right)$ and $\widehat{\ell_1}=\ell_{1} \left(1+B\right)^{2}$.
\end{lem}
\begin{proof}
	By the definition of $\mathcal{L}^{k}(\cdot)$, we have
	\begin{equation*}
		\begin{aligned}
			\left|\mathcal{L}^{k}(x,y)-\mathcal{L}^{k}(x^{\prime},y^{\prime})\right|=&\left|\underset{\widehat{\epsilon}^{k}}{\mathbb{E}}\left[l(x,y,m_{k}(x, \widehat{\epsilon}^{k}))\right]-\underset{\widehat{\epsilon}^{k}}{\mathbb{E}}\left[l(x^{\prime},y^{\prime},m_k(x^{\prime},\widehat{\epsilon}^{k}))\right]\right|\\
			\leq& \underset{\widehat{\epsilon}^{k}}{\mathbb{E}}\left|l(x,y,m_{k}(x, \widehat{\epsilon}^{k}))- l(x^{\prime},y^{\prime},m_k(x^{\prime},\widehat{\epsilon}^{k})\right|\\
			\leq&L_{1}\left(\|x-x^{\prime}\|+\|y-y^{\prime}\|+\|\left(\widehat{B}^{k, 1}\right)^{\top}(x-x^{\prime})\|\right)\\
			\leq&L_{1}\left(1+B\right)\left(\left\|x-x^{\prime} \right\|+\left\|y-y^{\prime} \right\|\right),
		\end{aligned}
	\end{equation*}
	where the first inequality follows from Jensen's inequality and the second inequality follows from $L_1$-Lipschitz continuity of $l(\cdot)$.

	For smoothness of $\mathcal{L}^{k}(\cdot)$, 
	\begin{equation}
		\begin{aligned}
			&\nabla_{x}\mathcal{L}^{k}(x,y)=\underset{\widehat{\epsilon}^{k}}{\mathbb{E}}\left[\nabla_{1}l(x,y,m_{k}(x,\widehat{\epsilon}^{k}))+(\widehat{B}^{k, 1})\nabla_{3}l\left(x,y,m_{k}(x,\widehat{\epsilon}^{k})\right)\right],\\
			&\nabla_{y}\mathcal{L}^{k}(x,y)=\underset{\widehat{\epsilon}^{k}}{\mathbb{E}}\left[\nabla_{y}l(x,y,m_{k}(x,\widehat{\epsilon}^{k}))\right],
		\end{aligned}
	\end{equation}
	we have
	\begin{equation*}
		\begin{aligned}
			&\left\|\nabla_{x}\mathcal{L}^{k}(x,y)-\nabla_{x}\mathcal{L}^{k}(x^{\prime},y^{\prime})\right\|\\
			\leq&\underset{\widehat{\epsilon}^{k}}{\mathbb{E}}\left\|\nabla_{1}l(x,y,m_{k}(x,\widehat{\epsilon}^{k})) - \nabla_{1}l(x^{\prime},y^{\prime},m_{k}(x^{\prime},\widehat{\epsilon}^{k})) \right\|\\
			&+\underset{\widehat{\epsilon}^{k}}{\mathbb{E}}\left\|(\widehat{B}^{k, 1})\nabla_{3}l\left(x,y,m_{k}(x,\widehat{\epsilon}^{k})\right)- (\widehat{B}^{k, 1})\nabla_{3}l(x^{\prime},y^{\prime},m_{k}(x^{\prime},\widehat{\epsilon}^{k})) \right\|\\
			\leq&\underset{\widehat{\epsilon}^{k}}{\mathbb{E}}\left\|\nabla_{1}l(x,y,m_{k}(x,\widehat{\epsilon}^{k})) - \nabla_{1}l(x^{\prime},y^{\prime},m_{k}(x^{\prime},\widehat{\epsilon}^{k})) \right\|\\
			&+\left\|\widehat{B}^{k, 1}\right\|_F \underset{\widehat{\epsilon}^{k}}{\mathbb{E}}\left\|\nabla_{3}l\left(x,y,m_{k}(x,\widehat{\epsilon}^{k})\right)- \nabla_{3}l(x^{\prime},y^{\prime},m_{k}(x^{\prime},\widehat{\epsilon}^{k})) \right\|\\
			\leq& \ell_{1}\left(1+B\right)^{2}\left\|(x,y)-(x^{\prime},y^{\prime})\right\|,
		\end{aligned}
	\end{equation*}
	where the first inequality follows from triangle inequality and Jensen's inequality, the second inequality follows from Cauchy-Schwarz inequality and the last inequality follows from $\ell_1$-smoothness of $l(\cdot)$.
	
	Similarly,
	\begin{equation*}\label{equ:lemma2_grady}
		\begin{aligned}
			&\left\|\nabla_{y}\mathcal{L}^{k}(x,y)-\nabla_{y}\mathcal{L}^{k}(x^{\prime},y^{\prime})\right\|
			\\
			\leq&\underset{\widehat{\epsilon}^{k}}{\mathbb{E}}\left\|\nabla_{y}l(x,y,m_{k}(x,\widehat{\epsilon}^{k})) - \nabla_{y}l(x^{\prime},y^{\prime},m_{k}(x^{\prime},\widehat{\epsilon}^{k}))\right\|\\
			\leq& \ell_{1}\left(1+B\right)\left\|(x,y)-(x^{\prime},y^{\prime})\right\|.
		\end{aligned}
	\end{equation*}
	
	Then,
	\begin{equation*}
		\left\|\nabla_{y}\mathcal{L}^{k}(x,y)-\nabla_{y}\mathcal{L}^{k}(x^{\prime},y^{\prime})\right\|\leq\ell_{1} \left(1+B\right)^{2}
\left\|(x,y)-(x^{\prime},y^{\prime})\right\|.		
	\end{equation*}
	The proof is complete.
\end{proof}
\begin{lem}\label{TR_Lemma2}
	Suppose Assumptions \ref{Ass:smoothness_local}--\ref{ass:stronglyconcave_local} hold, and there exist some constants $\delta_k, \tilde{\kappa}_{dcp} > 0$ such that
	\begin{equation}\label{TR_eq_descent_inequality}
		\Phi^{k}(\widehat{x}^{k})-\Phi^{k}(\widehat{x}^{k}+\tilde{s}^k)\geq \tilde{\kappa}_{dcp}\|\nabla\Phi^{k}(\widehat{x}^{k})\|\min\{\delta_k,1\}, \forall\widehat{x}^{k}\in\mathbb{R}^n
	\end{equation}
	with 
	\begin{equation*}\label{TR_eq_descent_inequality_sk}
		\begin{aligned}
			\tilde{s}^k = - \delta_{k} \nabla\Phi^{k}(\widehat{x}^k)/\|\nabla\Phi^{k}(\widehat{x}^k)\|.
		\end{aligned}
	\end{equation*}
	Suppose also that the inexact solution $y^{i,*}(x)$ satisfies
	$$
	\|y^{i,*}(x)-y^{k,*}(x)\|\leq \epsilon,\; \forall x \in \mathcal{B}(\widehat{x}^{k}, \delta_{k}),
	$$
	where $\epsilon \leq \min\left\{\frac{\|\nabla_{1} \mathcal{L}^{k}(\widehat{x}^{k},y^{i,*}(\widehat{x}^{k}))\|}{2\widehat{\ell}_{1}},\frac{\|\nabla\Phi^{k}(\widehat{x}^{k})\|\|\nabla_{1} \mathcal{L}^{k}(\widehat{x}^{k},y^{i,*}(\widehat{x}^{k}))\|\tilde{\kappa}_{dcp}}{8\widehat{L}_{1}\widehat{\ell}_{1}\max\{\delta_{max},1\}},\frac{\|\nabla_{1} \mathcal{L}^{k}(\widehat{x}^{k},y^{i,*}(\widehat{x}^{k}))\|\min\{\delta_k,1\}\tilde{\kappa}_{dcp}}{16\widehat{L}_1}\right\}$, and  $y^{k,*}(x)\coloneqq \arg\max_{y\in \mathcal{Y}}\mathcal{L}^k{(x,y)}$ is the exact solution.
 	Then, there exists $ \kappa_{dcp} > 0$ such that
	\begin{equation}
		\begin{aligned}
			\label{prop1-1}
			\mathcal{L}^{k}(\widehat{x}^k,y^{i,*}(\widehat{x}^k))-\mathcal{L}^{k}(\widehat{x}^k+s^k,y^{i,*}(\widehat{x}^k+s^k)) \geq \kappa_{dcp} \|\nabla_{1} \mathcal{L}^{k}(\widehat{x}^{k},y^{i,*}(\widehat{x}^{k}))\|\min \left\{\delta_{k}, 1\right\},
		\end{aligned}
	\end{equation}
	with
	\begin{equation*}
		\begin{aligned}
			\label{prop1-2}
			s^{k}= - \delta_{k} \nabla_{1} \mathcal{L}^{k}(\widehat{x}^{k},y^{i,*}(\widehat{x}^{k}))/\|\nabla_{1} \mathcal{L}^{k}(\widehat{x}^{k},y^{i,*}(\widehat{x}^{k}))\|
		\end{aligned}
	\end{equation*}
	and $\kappa_{dcp}=\frac{1}{8}\tilde{\kappa}_{dcp}.$
\end{lem}
\begin{proof}
	Recall the definitions of $\mathcal{L}^{k}(\cdot)$ and  $\Phi^{k}(\cdot)$ in \eqref{TR_notations} and \eqref{TR_trustRegionMinimization}, 
	\begin{equation*}
		\begin{aligned}
			& \Phi^{k}(\widehat{x}^k)-\mathcal{L}^{k}(\widehat{x}^{k},y^{i,*}(\widehat{x}^k))+\mathcal{L}^{k}\left(\widehat{x}^{k}+s^k,y^{i,*}(\widehat{x}^k+s^k)\right)-\Phi^{k}(\widehat{x}^{k}+s^k) \\
			\leq & \widehat{L}_1 \left\|y^{k,*}(\widehat{x}^k)-y^{i,*}(\widehat{x}^k)\right\|+
			\widehat{L}_1 \left\|y^{k,*}(\widehat{x}^k+s^k)-y^{i,*}(\widehat{x}^k+s^k)\right\|\\
			\leq & 2\widehat{L}_1 \epsilon,
		\end{aligned}
	\end{equation*}
	where the first inequality follows from $\widehat{L}_1$-Lipschitz continuity of $\mathcal{L}^{k}(\cdot)$ and the second inequality follows from the setting of the inexact solution $y^{i,*}(\cdot)$.
	
	Then,
	\begin{equation*}
		\begin{aligned}				&\mathcal{L}^{k}(\widehat{x}^k,y^{i,*}(\widehat{x}^k))-\mathcal{L}^{k}(\widehat{x}^k+s^k,y^{i,*}(\widehat{x}^k+s^k))\\
			\geq &  \Phi^{k}(\widehat{x}^k) - \Phi^{k}(\widehat{x}^{k}+s^k) - 2L_1 \epsilon\\
			\geq & \tilde{\kappa}_{dcp} \underbrace{\|\nabla\Phi^{k}(\widehat{x}^{k})\|}_{I_1} \min \{\delta_k, 1\} + \underbrace{\Phi^{k}(\widehat{x}^{k}+\tilde{s}^k) - \Phi^{k}(\widehat{x}^{k}+s^k)}_{I_2} - 2L_1 \epsilon,
		\end{aligned}
	\end{equation*}
	where the second inequality follows from \eqref{TR_eq_descent_inequality}.
	
	For $I_1$, by $\widehat{\ell}_1$-smoothness of $\mathcal{L}^{k}(\cdot)$,
	\begin{equation*}
		\begin{aligned}
			I_1
			&\geq \|\nabla_{1} \mathcal{L}^{k}(\widehat{x}^{k},y^{i,*}(\widehat{x}^{k}))\|-\widehat{\ell}_1 \|y^{k,*}(\widehat{x}^{k})-y^{i,*}(\widehat{x}^{k})\|\\
			&\geq \|\nabla_{1} \mathcal{L}^{k}(\widehat{x}^{k},y^{i,*}(\widehat{x}^{k}))\|- \widehat{\ell}_1 \epsilon.
		\end{aligned}
	\end{equation*}

	For $I_2$, since $\Phi^k(\cdot)$ is $\widehat{L}_1$-Lipschitz continuous by \cite[Lemma 4.7]{lin2020gradient}, we have
	\begin{equation*}
		\begin{aligned}
			I_2=&\Phi^k(\widehat{x}^k + s^k)-\Phi^k(\widehat{x}^k + \tilde{s}^{k})\\
			\leq& \widehat{L}_{1} \|s^k-\tilde{s}^{k}\|\\
			=&\widehat{L}_{1}\delta_k\left\|\frac{\nabla_{1} \mathcal{L}^{k}(\widehat{x}^{k},y^{i,*}(\widehat{x}^{k}))}{\|\nabla_{1} \mathcal{L}^{k}(\widehat{x}^{k},y^{i,*}(\widehat{x}^{k}))\|}-\frac{\nabla_{x} \Phi^{k}(\widehat{x}^{k})}{\|\nabla\Phi^{k}(\widehat{x}^{k})\|}\right\|\\
			=&\widehat{L}_{1}\delta_k  \left\|\frac{\|\nabla\Phi^{k}(\widehat{x}^{k})\| \nabla_{1} \mathcal{L}^{k}(\widehat{x}^{k},y^{i,*}(\widehat{x}^{k}))-\|\nabla_{1} \mathcal{L}^{k}(\widehat{x}^{k},y^{i,*}(\widehat{x}^{k}))\|\nabla_{x} \Phi^{k}(\widehat{x}^{k})}{\|\nabla_{1} \mathcal{L}^{k}(\widehat{x}^{k},y^{i,*}(\widehat{x}^{k}))\|\nabla\Phi^{k}(\widehat{x}^{k})}\right\| \\
			\leq& \widehat{L}_{1}\delta_k \frac{ \|\nabla_{1} \mathcal{L}^{k}(\widehat{x}^{k},y^{i,*}(\widehat{x}^{k}))\|\left(\|\nabla\Phi^{k}(\widehat{x}^{k})\|- \|\nabla_{1} \mathcal{L}^{k}(\widehat{x}^{k},y^{i,*}(\widehat{x}^{k}))\|\right)}{\|\nabla_{1} \mathcal{L}^{k}(\widehat{x}^{k},y^{i,*}(\widehat{x}^{k}))\|\|\nabla\Phi^{k}(\widehat{x}^{k})\|}\\
			&+\widehat{L}_{1}\delta_k  \frac{\|\nabla_{1} \mathcal{L}^{k}(\widehat{x}^{k},y^{i,*}(\widehat{x}^{k}))\| \| \nabla_{1} \mathcal{L}^{k}(\widehat{x}^{k},y^{i,*}(\widehat{x}^{k}))-\nabla\Phi^{k}(\widehat{x}^{k})\|}{{\|\nabla_{1} \mathcal{L}^{k}(\widehat{x}^{k},y^{i,*}(\widehat{x}^{k}))\|\|\nabla\Phi^{k}(\widehat{x}^{k})\|}}\\
			\leq & \widehat{L}_{1}\delta_k \frac{2\|\nabla_{1} \mathcal{L}^{k}(\widehat{x}^{k},y^{i,*}(\widehat{x}^{k}))\|\| \nabla_{x} \mathcal{L}^{k}(\widehat{x}^{k},y^{k,*}(\widehat{x}^{k}))- \nabla_{1} \mathcal{L}^{k}(\widehat{x}^{k},y^{i,*}(\widehat{x}^{k}))\| }{\|\nabla_{1} \mathcal{L}^{k}(\widehat{x}^{k},y^{i,*}(\widehat{x}^{k}))\|\|\nabla\Phi^{k}(\widehat{x}^{k})\|}\\
			\leq & \frac{2\widehat{L}_{1}\widehat{\ell}_{1}\delta_k}{\|\nabla\Phi^{k}(\widehat{x}^{k})\|}\|y^{k,*}(\widehat{x}^{k})-y^{i,*}(\widehat{x}^{k})\|\\
			\leq & \frac{2\widehat{L}_{1}\widehat{\ell}_{1}\delta_k}{\|\nabla\Phi^{k}(\widehat{x}^{k})\|}\epsilon,
		\end{aligned}
	\end{equation*}
	where the second and third inequalities follow from triangle inequality, and the fourth inequality follows from $\widehat{\ell}_1$-smoothness of $\mathcal{L}^{k}(\cdot)$.
	
	Then
	\begin{equation*}
		\begin{aligned}				&\mathcal{L}^{k}(\widehat{x}^k,y^{i,*}(\widehat{x}^k))-\mathcal{L}^{k}(\widehat{x}^k+s^k,y^{i,*}(\widehat{x}^k+s^k))\\
			\geq & \tilde{\kappa}_{dcp} \left(\|\nabla_{1} \mathcal{L}^{k}(\widehat{x}^{k},y^{i,*}(\widehat{x}^{k}))\| - \ell_{1} \epsilon\right) \min \{\delta_k, 1\} - \frac{2\widehat{L}_{1}\widehat{\ell}_{1}\delta_k}{\|\nabla\Phi^{k}(\widehat{x}^{k})\|}\epsilon - 2\widehat{L}_1 \epsilon.
		\end{aligned}
	\end{equation*}
	Given 
	\begin{small}
	\begin{equation*}
	\epsilon \leq \min\left\{\frac{\|\nabla_{1} \mathcal{L}^{k}(\widehat{x}^{k},y^{i,*}(\widehat{x}^{k}))\|}{2\widehat{\ell}_{1}},\frac{\|\nabla\Phi^{k}(\widehat{x}^{k})\|\|\nabla_{1} \mathcal{L}^{k}(\widehat{x}^{k},y^{i,*}(\widehat{x}^{k}))\|\tilde{\kappa}_{dcp}}{8\widehat{L}_{1}\widehat{\ell}_{1}\max\{\delta_{max},1\}},\frac{\|\nabla_{1} \mathcal{L}^{k}(\widehat{x}^{k},y^{i,*}(\widehat{x}^{k}))\|\min\{\delta_k,1\}\tilde{\kappa}_{dcp}}{16\widehat{L}_1}\right\},
	\end{equation*}
	\end{small}
	inequality \eqref{prop1-1} holds.
	The proof is complete.
\end{proof}

\section{Convergence Analysis}\label{section3}
In this section, we establish the almost sure convergence of the sequence $\{\widehat{x}^k\}$ generated by Algorithm~\ref{TR:algorithmCL} to a stationary point of SMDD~\eqref{eq:f_minimax_regression}. We begin with two propositions that characterize the probabilistic accuracy of the trust-region model and the value estimates. 
In the following analysis, we denote the random process as
$\{\widehat{X}^{k}, S^{k}, \Delta_{k}, \widehat{\Phi}^{k}(\cdot), V_k, V_{k+1/2}, M_k(\cdot)\},$
with the realizations
$\{\widehat{x}^{k}, s^{k}, \delta_{k}, \Phi^{k}(\cdot), v_k, v_{k+1/2}, m_k(\cdot)\}$, $\mathcal{F}_{k-1}$ as the $\sigma$-algebra generated by
$\{\widehat{\Phi}^{i}(\cdot), V_i, V_{i+1/2}\}_{i=0}^{k-1}$, $\mathcal{F}_{k-1/2}$ as the $\sigma$-algebra generated by
$\{\widehat{\Phi}^{i}(\cdot)\}_{i=0}^{k}$ and $\{V_i, V_{i+1/2}\}_{i=0}^{k-1}$, $\mathbb{P}(\cdot \mid \mathcal{F}_{k-1})$ and $\mathbb{P}(\cdot \mid \mathcal{F}_{k-1/2})$ as the the corresponding conditional probabilities. 

\begin{prop}\label{TR_proposition_1}
	Let $\widehat{x}^k\in \mathbb{R}^{n}$ be the iterate at the $k$-th iteration of Algorithm \ref{TR:algorithmCL}, $\delta_{k}>0$ be the trust region radius,  $T_k=\{(x^{i},\omega^{i})\}$ be the data set of size $N_k$ in the algorithm for local linear regression and $\alpha\in (0,1)$ be a probability parameter. 
	Suppose Assumptions \ref{TR_ass_primal}--\ref{ass:stronglyconcave_local} hold and the set $\{x^i\}$ is strongly $\Lambda$-poised \footnote{Strongly $\Lambda$-poised condition on the data samples guarantees the uniqueness of linear regression and allows the number of samples to grow arbitrarily large. 
	For the definition and procedure to generate a strongly $\Lambda$-poised set, see\cite{liu2022coupled} and \cite[Algorithm~6.7]{Conn2000Trustregionmethods}.}. Then
	
	{\rm(i)} There exists a constant $\kappa_{ef}>0$ such that
	$$\mathbb{P}\left(\lvert\Phi(x)-\widehat{\Phi}^{k}(x)\rvert\geq \kappa_{ef}\Delta^{2}_{k}, \exists x\in \mathcal{B}(\widehat{X}^{k}, \Delta_{k})\mid \mathcal{F}_{k-1}\right)\leq \frac{\alpha}{2},$$
	if $N_k \geq \max\{\mathcal{O}(\delta_k^{-4}\kappa_{ef}^{-2}\alpha^{-1}),\mathcal{O}(\alpha^{-2})\}$.
	
	{\rm(ii)} There exists a constant $\kappa_{ed}>0$ such that
	$$\mathbb{P}\left(\|\nabla\Phi(x)\|-\|\nabla\widehat{\Phi}^{k}(x)\|\geq \kappa_{ed}\Delta_{k}, \exists x\in \mathcal{B}(\widehat{X}^{k}, \Delta_{k})\mid \mathcal{F}_{k-1}\right)\leq \frac{\alpha}{2},$$
	if $N_{k}\geq \max\{\mathcal{O}(\delta_{k}^{-4}\kappa_{ed}^{-2}),\mathcal{O}(\alpha^{-2}),
	\mathcal{O}(\delta_{k}^{-2}\kappa_{ed}^{-2}\alpha^{-1})\}$.
\end{prop}
\begin{proof}
	\textbf{Part (i):}
	For each $\left(x^{i}, \omega^{i}\right)$, denote $\epsilon^{i} \coloneqq \omega^{i}-\psi\left(x^{i}\right)$, $e^{k, i} \coloneqq \psi\left(x^{i}\right)-\left(\widehat{B}^{k, 1}\right)^{\top} x^{i}-\left(\widehat{B}^{k, 0}\right)^{\top}+\epsilon^{i}$.
	Then, for any $x \in \mathcal{B}\left(\widehat{x}^{k}, \delta_{k}\right),$
	\begin{equation*}
		\begin{aligned}
			\vert\Phi(x)-\Phi^{k}(x)\vert
			=&\vert\max_{y\in\mathcal{Y}}\mathbb{E}_{\tilde{\epsilon}}\left[l(x,y,\psi(x)+\tilde{\epsilon})\right] -\mathbb{E}_{\widehat{\epsilon}^{k}}[l(x,y^{k,*}(x),m^{k}(x,\widehat{\epsilon}^{k}))]
			\vert\\
			\leq & \vert\mathbb{E}_{\tilde{\epsilon}}\left[l(x,y^{*}(x),\psi(x)+\tilde{\epsilon})\right]
			-\frac{1}{N_k}\sum_{i=1}^{N_k}l(x,y^{*}(x),\psi(x)+\epsilon^{i})\vert\\
			&+\vert\frac{1}{N_k}\sum_{i=1}^{N_k}l(x,y^{*}(x),\psi(x)+\epsilon^{i}) -\frac{1}{N_k}\sum_{i=1}^{N_k}l(x,y^{s,*}(x),\psi(x)+\epsilon^{i})\vert\\
			&+\vert\frac{1}{N_k}\sum_{i=1}^{N_k}l(x,y^{s,*}(x),\psi(x)+\epsilon^{i}) -\frac{1}{N_k}\sum_{i=1}^{N_k}l(x,y^{s,*}(x),m^{k}(x,e^{k,i}))\vert\\
			&+\vert \frac{1}{N_k}\sum_{i=1}^{N_k}l(x,y^{s,*}(x),m^{k}(x,e^{k,i}))-\frac{1}{N_k}\sum_{i=1}^{N_k}l(x,y^{k,*}(x),m^{k}(x,e^{k,i}))\vert,
		\end{aligned}
	\end{equation*}
	where $y^{s,*}(x)=\arg\underset{y\in \mathcal Y}{\max}\frac{1}{N_k}\sum_{i=1}^{N_k}l(x,y,\psi(x)+\epsilon^{i})$
	and the inequality follows from triangle inequality.
	By $L_1$-Lipschitz continuity of $l(\cdot)$,
	\begin{equation}\label{TR_prop_Phi_k_0}
		\begin{aligned}
			\vert\Phi(x)-\Phi^{k}(x)\vert \leq & \vert\tau_{k}(x)\vert+\frac{1}{N_k}\sum_{i=1}^{N_k}L_{1}\left\|\psi(x)+\epsilon^{i}-m^{k}(x,e^{k,i})\right\|\\
			&+L_{1}\|y^{*}(x)-y^{s,*}(x)\|+L_{1}\|y^{s,*}(x)-y^{k,*}(x)\|\\
			\leq & \vert\tau_{k}(x)\vert+
			2L_{1}\max_{x\in\mathcal{B}(\widehat{x}^{k},\delta^{k})}\left\|\psi(x)-(\widehat{B}^{k,1})^{\top}x
			-(\widehat{B}^{k,0})^{\top}\right\|\\
			&+L_{1}\underbrace{\|y^{*}(x)-y^{s,*}(x)\|}_{I_1}+L_{1}\underbrace{\|y^{s,*}(x)-y^{k,*}(x)\|}_{I_2},
		\end{aligned}
	\end{equation}		
	where $\tau_{k}(x)\coloneqq\mathbb{E}_{\tilde{\epsilon}}\left[l(x,y^{*}(x),\psi(x)+\tilde{\epsilon})\right]
	-\frac{1}{N_k}\sum_{i=1}^{N_k}l(x,y^{*}(x),\psi(x)+\epsilon^{i})$.
	
	For $I_1$, by the definition of $y^{s,*}(x)$,
	\begin{equation}\label{TR_prop_Phi_k_1}
		\begin{aligned}
			(y^{*}(x)-y^{s,*}(x))^{\top}\frac{1}{N_k}\sum_{i=1}^{N_k}\nabla_{y}l(x, y^{s,*}(x), \psi(x) + \epsilon^{i})\leq 0,
		\end{aligned}
	\end{equation}
	where $y^{*}(x)$ is the optimal solution.
	Similarly, by the optimality of $y^{*}(x)$,
	\begin{equation}\label{TR_prop_Phi_k_2}
		\begin{aligned}
			(y^{s,*}(x)-y^{*}(x))^{\top}\mathbb{E}_{\tilde{\epsilon}}\left[\nabla_{y}l(x, y^{*}(x), \psi(x) + \tilde{\epsilon})\right]\leq 0.
		\end{aligned}
	\end{equation}
	Then,
	$$(y^{*}(x)-y^{s,*}(x))^{\top}\left (\frac{1}{N_k}\sum_{i=1}^{N_k}\nabla_{y}l(x, y^{s,*}(x), \psi(x) + \epsilon^{i})-\mathbb{E}_{\tilde{\epsilon}}\left[\nabla_{y}l(x, y^{*}(x),\psi(x) + \tilde{\epsilon})\right]\right )\leq 0.$$
	Since $l(x, \cdot, \psi(x) + \epsilon^{i})$ is $\mu$-strongly concave, we have
	\begin{small}
		\begin{equation*}
			\left(y^{*}(x)-y^{s,*}(x)\right)^{\top}\frac{1}{N_k}\sum_{i=1}^{N_k}\left[\nabla_{y}l(x, y^{*}(x), \psi(x) + \epsilon^{i})-\nabla_{y}l(x, y^{s,*}(x), \psi(x) + \epsilon^{i})\right]
			+\mu\left\|y^{*}(x)-y^{s,*}(x)\right\|^{2}\leq 0.
		\end{equation*}
	\end{small}
	Combining the above two inequalities, we have
	\begin{small}
		\begin{equation*}
			\begin{aligned}
				\left(y^{*}(x)-y^{s,*}(x)\right)^{\top}\left (\frac{1}{N_k}\sum_{i=1}^{N_k}\nabla_{y}l(x, y^{*}(x), \psi(x) + \epsilon^{i})-\mathbb{E}_{\tilde{\epsilon}}\left[\nabla_{y}l(x, y^{*}(x),\psi(x) + \tilde{\epsilon})\right]\right )
				+\mu\left\|y^{*}(x)-y^{s,*}(x)\right\|^{2}\leq 0,
			\end{aligned}
		\end{equation*}
	\end{small}
	which implies,
	\begin{equation}\label{TR_prop_Phi_k_3}
		\begin{aligned}
			\left\|y^{*}(x)-y^{s,*}(x)\right\|&\leq \frac{1}{\mu}
			\underbrace{\left\|\frac{1}{N_k}\sum_{i=1}^{N_k}\nabla_{y}l(x,y^{*}(x),\psi(x)+\epsilon^{i})-
			\mathbb{E}_{\tilde{\epsilon}}\left[\nabla_{y}l(x,y^{*}(x),\psi(x)+\tilde{\epsilon})\right] \right\|}_{\|\xi_{k}(x)\|},
		\end{aligned}
	\end{equation}
	where the inequality follows from  Cauchy-Schwarz inequality.
	
	By a similar analysis,
	\begin{equation}\label{TR_prop_Phi_k_4}
		\begin{aligned}
			I_2=\left\|y^{k,*}(x)-y^{s,*}(x)\right\|\leq \frac{2\ell_{1}}{\mu}\max_{x\in\mathcal{B}(\widehat{x}^{k},\delta_{k})}
			\left\|\psi(x)-(\widehat{B}^{k,1})^{\top}x-(\widehat{B}^{k,0})^{\top} \right\|.
		\end{aligned}
	\end{equation}
	
	Plugging \eqref{TR_prop_Phi_k_3} and \eqref{TR_prop_Phi_k_4} into inequality \eqref{TR_prop_Phi_k_0}, we have
	\begin{equation*}
		\begin{aligned}
			\vert\Phi(x)-\Phi^{k}(x)\vert
			&\leq |\tau_{k}(x)|+
			\frac{L_{1}}{\mu}\|\xi_{k}(x)\|+2L_{1}\left(1+\frac{\ell_{1}}{\mu}\right)\max_{x\in\mathcal{B}(\widehat{x}^{k},\delta_{k})}
			\left\|\psi(x)-(\widehat{B}^{k,1})^{\top}x-(\widehat{B}^{k,0})^{\top} \right\|.\\
		\end{aligned}
	\end{equation*}
	
	In what follows, we establish the high-probability accuracy of the trust-region model $\Phi^{k}(\cdot)$.
	Consider the stochastic process $\widehat{\Phi}^{k}(\cdot)$ given $\mathcal{F}_{k-1}$,
	\begin{equation}\label{TR_prop_Phi_k_5}
		\begin{aligned}
			&\mathbb{P}\left(\lvert\Phi(x)-\widehat{\Phi}^{k}(x)\rvert\geq \kappa_{ef}\Delta^{2}_{k}, \exists x\in \mathcal{B}(\widehat{x}^{k}, \Delta_{k})\mid \mathcal{F}_{k-1}\right)\\
			\leq &\mathbb{P}\left(\vert\tau_{k}(x)\vert\geq \frac{1}{3}\kappa_{ef}\Delta^{2}_{k}, \exists x\in \mathcal{B}(\widehat{x}^{k}, \Delta_{k})\mid \mathcal{F}_{k-1}\right)\\
			&+ \mathbb{P}\left(\frac{L_{1}}{\mu}\|\xi_k(x)\|\geq \frac{1}{3}\kappa_{ef}\Delta^{2}_{k}, \exists x\in \mathcal{B}(\widehat{x}^{k}, \Delta_{k})\mid \mathcal{F}_{k-1}\right)\\
			&+\mathbb{P}\left(2L_{1}\left(1+\frac{\ell_{1}}{\mu}\right)\left\|\psi(x)-(\widehat{B}^{k,1})^{\top}x
			-(\widehat{B}^{k,0})^{\top}\right\|\geq \frac{1}{3}\kappa_{ef}\Delta^{2}_{k},  \exists x\in \mathcal{B}(\widehat{x}^{k}, \Delta_{k})\mid \mathcal{F}_{k-1}\right).
		\end{aligned}
	\end{equation}
	For $\tau_k(x)$, since $\{\epsilon^{i}\}_{i=1}^{N_k}$ are i.i.d. samples of $\tilde{\epsilon}$, the Berry-Esseen theorem \cite{liu2022coupled} implies that with the rate $\mathcal{O}(1/\sqrt{N_k})$, we have
	$\sqrt{N_{k}} \tau_{k}(x) \stackrel{d}{\rightarrow} H\sim\mathcal{N}\left(0, V_{l}(x)\right)$,
	where $V_{l}(x)\coloneqq\Var_{\tilde{\epsilon}}\left(l(x,y^{*}(x),\psi(x)+\tilde{\epsilon})\right)$,
	and $\Psi(\cdot)$ is the cumulative probability distribution of $\vert H\vert$. Then, for any $\alpha\in (0,1)$, there exist constants $\kappa_{ef}, C>0$ such that
	\begin{equation*}
		\begin{aligned}
			&\mathbb{P}\left(\vert\tau_{k}(x)\vert\geq \frac{1}{3}\kappa_{ef}\Delta^{2}_{k}, \exists x\in \mathcal{B}(\widehat{x}^{k}, \Delta_{k})\mid \mathcal{F}_{k-1}\right)\\
			\leq& \mathbb{P}\left(\frac{1}{\sqrt{N_k}}\vert H \vert\geq \frac{1}{3}\kappa_{ef}\Delta^{2}_{k}, \exists x\in \mathcal{B}(\widehat{x}^{k}, \Delta_{k})\mid \mathcal{F}_{k-1}\right)+CN_k^{-1/2}\\
			\leq & \frac{\alpha}{6},
		\end{aligned}
	\end{equation*}
	if $N_k\geq\max\{9(\Psi^{-1}(1-\frac{\alpha}{12}))^{2}\kappa_{ef}^{-2}\delta_{k}^{-4} ,144C^{2}\alpha^{-2}\}$.
	Similarly, 
	\begin{equation*}
		\begin{aligned}
			\mathbb{P}\left(\frac{L_{1}}{\mu}\|\xi_k(x)\|\geq \frac{1}{3}\kappa_{ef}\Delta^{2}_{k}, \exists x\in \mathcal{B}(\widehat{x}^{k}, \Delta_{k})\mid \mathcal{F}_{k-1}\right)
			\leq \frac{\alpha}{6},
		\end{aligned}
	\end{equation*}
	if
	$N_k\geq\max\{\mathcal{O}(\kappa_{ef}^{-2}\delta_{k}^{-4}) ,\mathcal{O}(\alpha^{-2})\}$.
	By \cite[Proposition 1 (b)]{liu2022coupled}, for any $\alpha\in (0,1)$, there exists a constant $\kappa_{ef}>0$ such that
	$$\mathbb{P}\left(2L_{1}\left(1+\frac{\ell_{1}}{\mu}\right)\left\|\psi(x)-(\widehat{B}^{k,1})^{\top}x
	-(\widehat{B}^{k,0})^{\top}\right\|\geq \frac{1}{3}\kappa_{ef}\Delta^{2}_{k}, \exists x\in \mathcal{B}(\widehat{x}^{k}, \Delta_{k})\mid \mathcal{F}_{k-1}\right)\leq \frac{\alpha}{6},
	$$
	if $N_k \geq \mathcal{O}(\delta_k^{-4}\kappa_{ef}^{-2}\alpha^{-1})$.
	
	Summarizing the terms, we obtain
	$$\mathbb{P}\left(\lvert\Phi(x)-\Phi^{k}(x)\rvert\geq \kappa_{ef}\Delta^{2}_{k}, \exists x\in \mathcal{B}(\widehat{x}^{k}, \Delta_{k})\mid \mathcal{F}_{k-1}\right)\leq \frac{\alpha}{2},$$ if
	$N_k \geq \max\{\mathcal{O}(\delta_k^{-4}\kappa_{ef}^{-2}\alpha^{-1}),\mathcal{O}(\alpha^{-2})\}$.

	\textbf{Part (ii):} By \cite[Lemma 4.3]{lin2020gradient}, $\Phi(x)$ and $\Phi^{k}(x)$ are differential,
	\begin{equation*}
		\begin{aligned}
			&\nabla\Phi(x)=\nabla_{1}\mathbb{E}_{\tilde{\epsilon}}\left[l(x, y^{*}(x), \psi(x) + \tilde{\epsilon})\right]+(\nabla\psi(x))^{\top}\nabla_{3}\mathbb{E}_{\tilde{\epsilon}}\left[l(x, y^{*}(x), \psi(x) + \tilde{\epsilon})\right],\\
			&\nabla\Phi^{k}(x)=\nabla_{1}\mathbb{E}_{\widehat{\epsilon}^{k}}\left[l(x, y^{k,*}(x), m_{k}(x,\widehat{\epsilon}^{k}))\right]+(\widehat{B}^{k,1})\nabla_{3}\mathbb{E}_{\widehat{\epsilon}^{k}}\left[l(x, y^{k,*}(x), m_{k}(x,\widehat{\epsilon}^{k}))\right],
		\end{aligned}
	\end{equation*}
	where $y^{*}(x)=\arg\max_{y\in \mathcal Y} \mathcal{L}(x,y)$, $y^{k,*}(x)=\arg\max_{y\in \mathcal Y} \mathcal{L}^{k}(x,y)$.
	
	Then, by the triangle inequality, we have
	\begin{equation*}
		\begin{aligned}
			&\left\|\nabla\Phi(x)\right\|-\left\|\nabla \Phi^{k}(x)\right\|\\
			\leq&\left\|\nabla_{x}\mathbb{E}_{\tilde{\epsilon}}\left[l(x, y^{*}(x), \psi(x) + \tilde{\epsilon})\right]
			-\nabla_{x}\mathbb{E}_{\widehat{\epsilon}^{k}}\left[l(x, y^{k,*}(x), m_{k}(x,\widehat{\epsilon}^{k}))\right]\right\|\\
			\leq&\left\|\nabla_{x}\mathbb{E}_{\tilde{\epsilon}}\left[l(x, y^{*}(x), \psi(x) + \tilde{\epsilon})\right]
			-\frac{1}{N_k}\sum_{i=1}^{N_k}\nabla_{x}\left[l(x, y^{*}(x), \psi(x) + \epsilon^{i})\right]\right\|\\
			&+\left\|\frac{1}{N_k}\sum_{i=1}^{N_k}\nabla_{x}\left[l(x, y^{*}(x), \psi(x) + \epsilon^{i})\right]
			-\nabla_{x}\mathbb{E}_{\widehat{\epsilon}^{k}}\left[l(x, y^{k,*}(x), m_{k}(x,\widehat{\epsilon}^{k}))\right]\right\|\\
			\leq&\underbrace{\left\|\nabla_{1}\mathbb{E}_{\tilde{\epsilon}}\left[l(x, y^{*}(x), \psi(x) + \tilde{\epsilon})\right]
				-\frac{1}{N_k}\sum_{i=1}^{N_k}\nabla_{1}l(x, y^{*}(x), \psi(x) + \epsilon^{i})\right\|}_{\eta_{k1}(x)}\\
			&+\underbrace{\left\|(\nabla\psi(x))^{\top}\nabla_{3}\mathbb{E}_{\tilde{\epsilon}}\left[l(x, y^{*}(x), \psi(x) + \tilde{\epsilon})\right]-(\nabla\psi(x))^{\top}\frac{1}{N_k}\sum_{i=1}^{N_k}\nabla_{3}\left[l(x, y^{*}(x), \psi(x) + \epsilon^{i})\right]\right\|}_{\rho_{k1}(x)}\\
			&+\underbrace{\left\|\frac{1}{N_k}\sum_{i=1}^{N_k}\nabla_{1}l(x, y^{*}(x), \psi(x) + \epsilon^{i})
				-\frac{1}{N_k}\sum_{i=1}^{N_k}\left[\nabla_{1}l(x, y^{k,*}(x), m_{k}(x,e^{k,i}))\right]\right\|}_{\eta_{k2}(x)}\\
			&+\underbrace{\left\|(\nabla\psi(x))^{\top}\frac{1}{N_k}\sum_{i=1}^{N_k}\nabla_{3}\left[l(x, y^{*}(x), \psi(x) + \epsilon^{i})\right]
				-\widehat{B}^{k,1}\frac{1}{N_k}\sum_{i=1}^{N_k}\nabla_{3}\left[l(x, y^{k,*}(x), m_{k}(x,e^{k,i}))\right]\right\|}_{\rho_{k2}(x)}.
		\end{aligned}
	\end{equation*}
	Next, we bound the terms on the right-hand side of the above inequality. By the Berry-Esseen theorem, we can derive that $\sqrt{N_k}\eta_{k1}(x)$ and $\sqrt{N_k}\rho_{k1}(x)$ converge in distribution to a normal distribution with the rate $\mathcal{O}(1/\sqrt{N_k})$ uniformly for all $x \in \mathcal{B}(\widehat{x}^{k}, \delta_k)$.
	For any $x\in\mathcal{B}(\widehat{x}^{k},\delta_k)$,
	\begin{equation*}
		\begin{aligned}
			\eta_{k2}(x)
			\leq &\frac{1}{N_k}\sum_{i=1}^{N_k}\left\|\nabla_{1}l(x, y^{*}(x), \psi(x) + \epsilon^{i})-
			\nabla_{1}l(x, y^{*}(x), m_{k}(x,e^{k,i}))\right\|\\
			&+\frac{1}{N_k}\sum_{i=1}^{N_k}\left\|\nabla_{1}l(x, y^{*}(x), m_{k}(x,e^{k,i}))-\nabla_{1}l(x, y^{k,*}(x), m_{k}(x,e^{k,i}))\right\|\\
			\leq &\ell_{1}\frac{1}{N_k}\sum_{i=1}^{N_k}\left\|\psi(x) + \epsilon^{i}-m_{k}(x,e^{k,i}) \right\|
			+\ell_{1}\frac{1}{N_k}\sum_{i=1}^{N_k}\left\|y^{*}(x)-y^{k,*}(x) \right\|\\
			\leq & 2\ell_{1}\max_{x\in\mathcal{B}(\widehat{x}^{k},\delta_{k})}\left\|\psi(x)-\left(\widehat{B}^{k, 1}\right)^{\top} x-\left(\widehat{B}^{k, 0}\right)^{\top} \right\|+\ell_{1}\left\|y^{*}(x)-y^{k,*}(x) \right\|\\
			\leq & 2\ell_{1}(1+\frac{\ell_{1}}{\mu})\max_{x\in\mathcal{B}(\widehat{x}^{k},\delta_{k})}
			\left\|\psi(x)-(\widehat{B}^{k,1})^{\top}x-(\widehat{B}^{k,0})^{\top} \right\|+\frac{\ell_{1}}{\mu}\|\xi_{k}(x)\|,\\
		\end{aligned}
	\end{equation*}
	where the first inequality follows from triangle inequality and Jensen's inequality, the second inequality follows from Assumption \ref{Ass:smoothness_local}, and the last inequality follows from inequalities \eqref{TR_prop_Phi_k_3} and \eqref{TR_prop_Phi_k_4}.
	\begin{small}
		\begin{equation*}
			\begin{aligned}
				\rho_{k2}(x)
				\leq &\left\|
				\left(\widehat{B}^{k,1}\right)\frac{1}{N_k}\sum_{i=1}^{N_k}\nabla_{3}\left[l(x, y^{k,*}(x), m_{k}(x,e^{k,i}))\right]
				-(\nabla\psi(x))^{\top}\frac{1}{N_k}\sum_{i=1}^{N_k}\nabla_{3}\left[l(x, y^{k,*}(x), m_{k}(x,e^{k,i}))\right]\right\|\\
				&+\left\|(\nabla\psi(x))^{\top}\frac{1}{N_k}\sum_{i=1}^{N_k}\nabla_{3}\left[l(x, y^{k,*}(x), m_{k}(x,e^{k,i}))\right]
				-(\nabla\psi(x))^{\top}\frac{1}{N_k}\sum_{i=1}^{N_k}\nabla_{3}\left[l(x, y^{*}(x), m_{k}(x,e^{k,i})\right]\right\|\\
				&+\left\|(\nabla\psi(x))^{\top}\frac{1}{N_k}\sum_{i=1}^{N_k}\nabla_{3}\left[l(x, y^{*}(x), m_{k}(x,e^{k,i}))\right]
				-(\nabla\psi(x))^{\top}\frac{1}{N_k}\sum_{i=1}^{N_k}\nabla_{3}\left[l(x, y^{*}(x), \psi(x) + \epsilon^{i})\right]\right\|\\
			\end{aligned}
		\end{equation*}
		\begin{equation*}
			\begin{aligned}
				\leq &L_{1}\left\|\widehat{B}^{k,1}
				-\nabla\psi(x))\right\|+\ell_{1}\left\|\nabla\psi(x))\right\|\left\|y^{*}(x)-y^{k,*}(x)\right\|\\
				&+
				2\ell_{1}\left\|\nabla\psi(x))\right\|\max_{x\in\mathcal{B}(\widehat{x}^{k},\Delta_{k})}\left\|\psi(x)-\left(\widehat{B}^{k, 1}\right)^{\top} x-\left(\widehat{B}^{k, 0}\right)^{\top} \right\|\\
				\leq &L_{1}\left\|\widehat{B}^{k,1}
				-\nabla\psi(x))\right\|+L_{0}\ell_{1}\left\|y^{*}(x)-y^{k,*}(x)\right\|
				+
				2L_{0}\ell_{1}\max_{x\in\mathcal{B}(\widehat{x}^{k},\Delta_{k})}\left\|\psi(x)-\left(\widehat{B}^{k, 1}\right)^{\top} x-\left(\widehat{B}^{k, 0}\right)^{\top} \right\|\\
				\leq &L_{1}\left\|\widehat{B}^{k,1}
				-\nabla\psi(x))\right\|+\frac{L_{0}\ell_{1}}{\mu}\left|\xi_{k}(x)\right|+ 2L_{0}\ell_{1}(1+\frac{\ell_{1}}{\mu})\max_{x\in\mathcal{B}(\widehat{x}^{k},\Delta_{k})}\left\|\psi(x)-\left(\widehat{B}^{k, 1}\right)^{\top} x-\left(\widehat{B}^{k, 0}\right)^{\top} \right\|,
			\end{aligned}
		\end{equation*}
	\end{small}where the second inequality follows from $L_1$-Lipschitz continuity and $\ell_{1}$-smoothness of $l(\cdot)$, the third inequality follows from $L_0$-Lipschitz continuity of $\psi(\cdot)$ and the last inequality follows from \eqref{TR_prop_Phi_k_3} and \eqref{TR_prop_Phi_k_4}.
	
	 Combining the above analysis with \cite[Proposition 1]{liu2022coupled}, we have there exists a constant $\kappa_{ed}>0$ such that
	\begin{equation*}
		\begin{aligned}
			& \mathbb{P}\left(\left\|\nabla\Phi(x)\right\|-\|\nabla\widehat{\Phi}^{k}(x)\| \geq \kappa_{e d} \Delta_{k}, \exists x \in \mathcal{B}\left(\widehat{x}^{k}, \Delta_{k}\right)\mid \mathcal{F}_{k-1}\right) \\
			\leq & \mathbb{P}\left(\|\eta_{k1}(x)\|+\|\eta_{k2}(x)\|+\|\rho_{k1}(x)\|+\|\rho_{k2}(x)\| \geq \kappa_{e d} \Delta_{k}, \exists x \in \mathcal{B}\left(\widehat{x}^{k}, \Delta_{k}\right)\mid \mathcal{F}_{k-1}\right) \\
			\leq & \mathbb{P}\left(\|\eta_{k 1}(x)\| \geq \frac{1}{4} \kappa_{e d} \Delta_{k}, \exists x \in \mathcal{B}\left(\widehat{x}^{k}, \Delta_{k}\right)\mid \mathcal{F}_{k-1}\right)+\mathbb{P}\left(\|\eta_{k 2}(x)\| \geq \frac{1}{4} \kappa_{e d} \Delta_{k}, \exists x \in \mathcal{B}\left(\widehat{x}^{k}, \Delta_{k}\right)\mid \mathcal{F}_{k-1}\right) \\
			+&\mathbb{P}\left(\|\rho_{k 1}(x)\| \geq \frac{1}{4} \kappa_{e d} \Delta_{k}, \exists x \in \mathcal{B}\left(\widehat{x}^{k}, \Delta_{k}\right)\mid \mathcal{F}_{k-1}\right)+\mathbb{P}\left(\|\rho_{k 2}(x)\| \geq \frac{1}{4} \kappa_{e d} \Delta_{k}, \exists x \in \mathcal{B}\left(\widehat{x}^{k}, \Delta_{k}\right)\mid \mathcal{F}_{k-1}\right) \\
			\leq &\frac{\alpha}{2},
		\end{aligned}
	\end{equation*}
	 if $N_{k}\geq \max\{\mathcal{O}(\delta_{k}^{-4}\kappa_{ed}^{-2}),\mathcal{O}(\alpha^{-2}),
	\mathcal{O}(\delta_{k}^{-2}\kappa_{ed}^{-2}\alpha^{-1})\}$.
	The proof is complete.
\end{proof}
Proposition \ref{TR_proposition_1} shows that with enough number of samples, the trust region model $\Phi^{k}(\cdot)$ is a probabilistically accurate approximation of the primal function $\Phi(\cdot)$ in the trust region.
Combining Proposition \ref{TR_proposition_1} (a) and (b), we conclude that if $N_k \geq \max\{\mathcal{O}(\delta_k^{-4}(1-\alpha)^{-1}),\mathcal{O}((1-\alpha)^{-2})\}$, the event
\begin{equation}\label{TR_notion_I_k}
	I_{k} = \left\{\|\nabla\Phi(x)\|-\|\nabla\widehat{\Phi}^{k}(x)\| \leq \kappa_{ed}\Delta_{k},\; \left|\Phi(x) -\widehat{\Phi}^{k}(x) \right| \leq \kappa_{ef}\Delta_{k}^{2}, \forall x \in \mathcal{B}(\widehat{X}^{k}, \Delta_{k})\right\}
\end{equation}
holds with the probability at least $\alpha$. 
$I_k$ will be used as a metric to characterize the accuracy of the approximation model $\Phi^{k}(\cdot)$.
\begin{prop}\label{TR_proposition_2}
	Suppose Assumptions \ref{TR_ass_primal}--\ref{ass:stronglyconcave_local} holds. Let $S_k = \{\omega^i\}_{i=1}^{|S_k|}$ and $S_{k+1/2} = \{\omega^j\}_{j=1}^{|S_{k+1/2}|}$ be two sample sets with $\omega^i \sim \mathcal{D}(\widehat{x}^k)$ and $\omega^j \sim \mathcal{D}(\widehat{x}^k + s^k)$, and let $\beta\in (0,1)$ be a probability parameter. Denote 
	\begin{equation*}\small
		\begin{aligned}
			&V_k = \frac{1}{|S_k|} \sum_{i=1}^{|S_k|} l\left( \widehat{x}^k, y^{is,*}(\widehat{x}^k), \omega^i \right), \quad\\
			&V_{k+1/2} = \frac{1}{|S_{k+1/2}|} \sum_{j=1}^{|S_{k+1/2}|} 
			l\left( \widehat{x}^k + s^k, y^{is,*}(\widehat{x}^k + s^k), \omega^j \right),
		\end{aligned}
	\end{equation*}
	with the inexact solution $y^{is,*}(x)$ satisfies
	$\|y^{s,*}(x) - y^{is,*}(x)\| < \epsilon$ for $x \in \{\widehat{x}^k, \widehat{x}^k + s^k\}$ and $ |S_k|, |S_{k+1/2}| \geq M_k$.
	Then for any $\beta\in (0,1)$ there exists $\epsilon_F > 0$ such that
	$$\mathbb{P}\left(J_k\mid \mathcal{F}_{k-1/2}\right)\geq \beta$$
	if
	$\epsilon\leq\dfrac{1}{3L_1} \epsilon_F \Delta_k^2$ and
	$M_k = \max\left\{ \mathcal{O}\left( \Delta_k^{-4} \epsilon_F^{-2} \right),  \mathcal{O}((1-\beta)^{-2}) \right\},$ where 
	\begin{equation}\label{TR_notion_J_k}
		J_k = \{|V_k - \Phi(\widehat{X}^k)| \leq \epsilon_F \Delta_k^2, \quad |V_{k+1/2} - \Phi(\widehat{X}^k + S^k)| \leq \epsilon_F \Delta_k^2  \text{\ for some \ } \Delta_k>0 \}.
	\end{equation}
\end{prop}
\begin{proof}
	By the definition of $\Phi(\widehat{x}^k)$ defined in \eqref{TR_notations}  and $v_k$ defined in \eqref{TR_v_k},
	\begin{equation*}
		\begin{aligned}
			\vartheta(\widehat{x}^{k})\coloneqq &v_k-\Phi(\widehat{x}^k)\\
			=&\underbrace{\frac{1}{\vert S_k\vert}\sum_{\omega_i\in S_k}l(\widehat{x}^{k}, y^{is,*}(\widehat{x}^{k}), \omega^{i})-\frac{1}{\vert S_k\vert}\sum_{\omega_i\in S_k}l(\widehat{x}^{k}, y^{s,*}(\widehat{x}^{k}), \omega^{i})}_{\vartheta_{1}(\widehat{x}^{k})}\\
			&+\underbrace{\frac{1}{\vert S_k\vert}\sum_{\omega_i\in S_k}l(\widehat{x}^{k}, y^{s,*}(\widehat{x}^{k}), \omega^{i})-\mathbb{E}_{\tilde{\epsilon}}\left[l(\widehat{x}^{k},y^{s,*}(\widehat{x}^{k}),\psi(\widehat{x}^{k})+\tilde{\epsilon})\right]}_{\vartheta_{2}(\widehat{x}^{k})}\\
			&+\underbrace{\mathbb{E}_{\tilde{\epsilon}}\left[l(\widehat{x}^{k},y^{s,*}(\widehat{x}^{k}),\psi(\widehat{x}^{k})+\tilde{\epsilon})\right]-\mathbb{E}_{\tilde{\epsilon}}\left[l(\widehat{x}^{k},y^{*}(\widehat{x}^{k}),\psi(\widehat{x}^{k})+\tilde{\epsilon})\right]}_{\vartheta_{3}(\widehat{x}^{k})}.
		\end{aligned}
	\end{equation*}
	
	For $\vartheta_{1}(\widehat{x}^{k})$, by $L_1$-Lipschitz continuity of $l(\cdot)$ and the accuracy of the inexact solution $y^{is,*}(\widehat{x}^{k})$, 
	$$\vert\vartheta_{1}(\widehat{x}^{k})\vert\leq L_{1}\left\|y^{s,*}(\widehat{x}^{k})-y^{is,*}(\widehat{x}^{k})\right\|< \frac{1}{3}\epsilon_{F} \delta^{2}_{k},$$
	where $y^{s,*}(\widehat{x}^{k})$ is the corresponding optimal solution.

	For $\vartheta_{2}(\widehat{x}^{k})$, the Berry-Esseen theorem implies that with the rate $\mathcal{O}(1/\sqrt{\vert S_k \vert})$, 
	$\sqrt{\vert S_k\vert} \vartheta_{2}(\widehat{x}^{k}) \stackrel{d}{\rightarrow} H\sim\mathcal{N}\left(0, V_{l}(\widehat{x}^{k})\right)$,
	where $V_{l}(\widehat{x}^{k})=\Var_{\tilde{\epsilon}}\left(l(\widehat{x}^{k},y^{s,*}(\widehat{x}^{k}),\psi(\widehat{x}^{k})+\tilde{\epsilon})\right)$, and $\Psi(\cdot)$ is the cumulative probability distribution of $|H|$. For any $\beta \in (0,1)$, there exist constants $\epsilon_{F}, C>0$ such that
	$$\mathbb{P}\left(\vert\vartheta_{2}(\widehat{x}^k)\vert\geq \frac{1}{3}\epsilon_{F}\Delta^{2}_{k}\mid \mathcal{F}_{k-1/2}\right)\leq  \frac{1-\beta}{3},$$
	 if $\vert S_k\vert\geq\max\{9(\Psi^{-1}(1-\frac{1-\beta}{6}))^{2}\epsilon_{F}^{-2}\Delta_{k}^{-4} ,36C^{2}(1-\beta)^{-2}\}$.
	
	For $\vartheta_{3}(\widehat{x}^k)$, by Assumption \ref{Ass:smoothness_local},
	$$\vert\vartheta_{3}(\widehat{x}^k)\vert\leq L_{1}\left\|y^{s,*}(\widehat{x}^k)-y^{*}(\widehat{x}^k)\right\|.$$
	By the definition of $y^{s,*}(\widehat{x}^k)$ and $y^{*}(\widehat{x}^k)$ along with Assumption \ref{ass:stronglyconcave_local},
	\begin{equation*}
		\begin{aligned}
			\left\|y^{*}(\widehat{x}^k)-y^{s,*}(\widehat{x}^k)\right\|&\leq \frac{1}{\mu}
			\left\|\frac{1}{|S_k|}\sum_{i\in S_k}\nabla_{y}l(\widehat{x}^k,y^{*}(\widehat{x}^k),\psi(\widehat{x}^k)+\epsilon^{i})-
			\mathbb{E}_{\tilde{\epsilon}}\left[\nabla_{y}l(\widehat{x}^k,y^{*}(\widehat{x}^k),\psi(\widehat{x}^k)+\tilde{\epsilon})\right] \right\|.
		\end{aligned}
	\end{equation*}
	Then, by the Berry-Esseen theorem,
	\begin{equation*}
		\begin{aligned}
			\mathbb{P}\left(\vert\vartheta_{3}(x)\vert\geq \frac{1}{3}\epsilon_{F}\Delta^{2}_{k}\mid \mathcal{F}_{k-1/2}\right)
			\leq \frac{1-\beta}{3},
		\end{aligned}
	\end{equation*}
	if $\vert S_k\vert\geq\max\{(\mathcal{O}(\epsilon_{F}^{-2}\Delta_{k}^{-4}) ,\mathcal{O}((1-\beta)^{-2})\}$.
	
	Putting these pieces together, we have
	$$\mathbb{P}\left(\vert\vartheta(\widehat{x}^{k})\vert\leq \epsilon_{F}\Delta^{2}_{k}\mid \mathcal{F}_{k-1/2}\right)\geq \beta.$$	
	By a similar analysis on $V_{k+1/2}$, the result follows.
	The proof is complete.		
\end{proof}

Proposition \ref{TR_proposition_2} shows that with enough number of samples, $V_k,V_{k+1/2}$ are probabilistically accurate value estimates of $\Phi(\cdot)$ evaluated at $\widehat{x}^k$ and $\widehat{x}^k+s^{k}$, respectively. 
Similar to $I_k$ defined in \eqref{TR_notion_I_k}, $J_k$ will be used as a metric to characterize the accuracy of the value estimates.
\begin{lem}\label{TR_Lemma3}
	Suppose that
	\begin{itemize}
		\item[\rm(a)] The number of samples for local linear regression $$|T_k|\geq \max\{\mathcal{O}(\delta_k^{-4}(1-\alpha)^{-1}),\mathcal{O}((1-\alpha)^{-2})\}.$$ 
		\item[\rm(b)] The number of samples for value estimates $\max\{|S_{k}|, |S_{k+1/2}|\}\geq \max\left\{ \mathcal{O}\left( \delta_k^{-4} \epsilon_F^{-2} \right),  \mathcal{O}((1-\beta)^{-2}) \right\}$.
		\item[\rm(c)] The inexact solution $y^{is,*}(x)$ satisfies
		$\|y^{s,*}(x) - y^{is,*}(x)\| < \dfrac{1}{3L_1} \epsilon_F \delta_k^2$ for any $x\in\mathcal{B}(\widehat{x}^{k},\delta_k)$, and $\epsilon_F\leq \kappa_{ef}$.
		\item[\rm(d)] The trust-region radius satisfies
		\[
		\delta_k \leq \min \left\{\frac{1}{\eta_2}, \; \frac{\kappa_{dcp}(1 - \eta_1)}{6\kappa_{ef}\max\{\delta_{\max},1\}} \right\} \|\nabla_{1} \mathcal{L}^{k}(\widehat{x}^{k},y^{i,*}(\widehat{x}^{k}))\|,
		\]
		where $\eta_{1}$ and $\eta_{2}$ are parameters for acceptance in Algorithm \ref{TR:algorithmCL}, and the inexact maximizer $y^{i,*}(x)$ satisfies
		\[\|y^{i,*}(x)-y^{k,*}(x)\| \leq \frac{\kappa_{ef}\delta_k^{2}}{L_1}, \forall x\in \mathcal{B}(\widehat{x}^{k},\delta_k).\]
	\end{itemize}
	\noindent Then, the $k$-th iteration is successful with probability at least $\alpha\beta$.
\end{lem}
\begin{proof}
	Note that the $k$-th iteration is successful, if the actual-to-predicted reduction criterion $\rho_{k} \geq \eta_{1}$ in Step 8 of Algorithm \ref{TR:algorithmCL} holds. In the following, we focus on $\rho_k$.
	By the definition of $\rho_k$,
	\begin{small}
		\begin{equation}\label{TR_lemma_rhok}
			\begin{aligned}
				\rho_k &= \frac{v_k - v_{k+1/2}}{ \mathcal{L}^{k}(\widehat{x}^{k},y^{i,*}(\widehat{x}^{k})) -  \mathcal{L}^{k}(\widehat{x}^{k}+s^k,y^{i,*}(\widehat{x}^{k}+s^k))} \\
				&= \frac{v_k - \Phi(\widehat{x}^k)}{\mathcal{L}^{k}(\widehat{x}^{k},y^{i,*}(\widehat{x}^{k})) -  \mathcal{L}^{k}(\widehat{x}^{k}+s^k,y^{i,*}(\widehat{x}^{k}+s^k))} + \frac{\Phi(\widehat{x}^k) -  \mathcal{L}^{k}(\widehat{x}^{k},y^{i,*}(\widehat{x}^{k}))}{ \mathcal{L}^{k}(\widehat{x}^{k},y^{i,*}(\widehat{x}^{k})) -  \mathcal{L}^{k}(\widehat{x}^{k}+s^k,y^{i,*}(\widehat{x}^{k}+s^k))} \\
				&+ \frac{\mathcal{L}^{k}(\widehat{x}^{k},y^{i,*}(\widehat{x}^{k})) - \mathcal{L}^{k}(\widehat{x}^{k}+s^k,y^{i,*}(\widehat{x}^{k}+s^k))}{ \mathcal{L}^{k}(\widehat{x}^{k},y^{i,*}(\widehat{x}^{k})) -  \mathcal{L}^{k}(\widehat{x}^{k}+s^k,y^{i,*}(\widehat{x}^{k}+s^k))}+\frac{\mathcal{L}^{k}(\widehat{x}^{k}+s^k,y^{i,*}(\widehat{x}^{k}+s^k))-\Phi(\widehat{x}^k + s^k)}{ \mathcal{L}^{k}(\widehat{x}^{k},y^{i,*}(\widehat{x}^{k})) -  \mathcal{L}^{k}(\widehat{x}^{k}+s^k,y^{i,*}(\widehat{x}^{k}+s^k))}\\
				&+ \frac{\Phi(\widehat{x}^k + s^k) - v_{k+1/2}}{ \mathcal{L}^{k}(\widehat{x}^{k},y^{i,*}(\widehat{x}^{k})) -  \mathcal{L}^{k}(\widehat{x}^{k}+s^k,y^{i,*}(\widehat{x}^{k}+s^k))}.
			\end{aligned}
		\end{equation}
	\end{small}
	
	By $L_1$-Lipschitz continuity of $l(\cdot)$ and the setting of the inexact solution $y^{i,*}(x)$,
	$$|\Phi^{k}(x)-\mathcal{L}^{k}(x,y^{i,*}(x))|\leq  \kappa_{ef}\delta_k^{2}, \forall x \in \mathcal{B}(\widehat{x}^k, \delta_k).$$
	If $I_k$ defined in \eqref{TR_notion_I_k} holds,
	\begin{equation}
		\label{lemma2-1}
		\begin{aligned}
			|\Phi(\widehat{x}^k) - \mathcal{L}^{k}(\widehat{x}^{k},y^{i,*}(\widehat{x}^{k}))| \leq 2\kappa_{ef} \delta^2_k, \;\;|\Phi(\widehat{x}^k + s^k) -  \mathcal{L}^{k}(\widehat{x}^{k}+s^k,y^{i,*}(\widehat{x}^{k}+s^k))| \leq 2\kappa_{ef} \delta^2_k.
		\end{aligned}
	\end{equation}
	
	If $J_k$ defined in \eqref{TR_notion_J_k} holds,  we have by \( \epsilon_F \leq \kappa_{ef} \) that
	\begin{equation}
		\label{lemma2-2}
		\begin{aligned}
			|v_k - \Phi(\widehat{x}^k)| \leq \kappa_{ef} \delta^2_k, \;\;|v_{k+1/2} - \Phi(\widehat{x}^k + s^k)| \leq \kappa_{ef} \delta^2_k.
		\end{aligned}
	\end{equation}
	
	Plugging inequalities (\ref{lemma2-1}) and (\ref{lemma2-2}) into \eqref{TR_lemma_rhok}, we have
	\[
	|\rho_k - 1| \leq \frac{6 \kappa_{ef} \max\{\delta_{\max},1\} \delta_k}{\kappa_{dcp} \|\nabla_{1} \mathcal{L}^{k}(\widehat{x}^{k},y^{i,*}(\widehat{x}^{k}))\|} \leq 1 - \eta_1,
	\]
	where the last inequality follows from the fact that \( \delta_k \leq \frac{\kappa_{dcp}(1 - \eta_1)}{6 \kappa_{ef}\max\{\delta_{\max},1\}} \|\nabla_{1} \mathcal{L}^{k}(\widehat{x}^{k},y^{i,*}(\widehat{x}^{k}))\| \). 
	Propositions \ref{TR_proposition_1} implies that $I_k$ holds with probabilities at least $\alpha$, and  Proposition \ref{TR_proposition_2} implies that $J_k$ holds with probabilities at least $\beta$. Then the $k$-th iteration is successful with probability at least $\alpha\beta$.
	The proof is complete.
\end{proof}

\begin{lem}\label{TR_Lemma4}
Suppose that
\begin{itemize}
	\item[\rm(a)] The number of samples for local linear regression $$|T_k|\geq \max\{\mathcal{O}(\delta_k^{-4}(1-\alpha)^{-1}),\mathcal{O}((1-\alpha)^{-2})\}.$$
	\item[\rm(b)] The trust-region radius satisfies
	\begin{equation}
		\label{TR_Lemma1_equ_1}
		\delta_k \leq \frac{1}{\left( \tfrac{8\kappa_{ef}}{\kappa_{dcp}} + 2\kappa_{ed} \right) \max \{ 1, \delta_{\max} \}} \|\nabla\Phi(\widehat{x}^k)\|.
	\end{equation}
	\item[\rm(c)]The inexact solution $y^{i,*}(x)$ satisfies
	\[
	\|y^{i,*}(x)-y^{k,*}(x)\| \leq \min\left\{\frac{\kappa_{ed}\delta_k}{\widehat{\ell}_1},\; \frac{\kappa_{ef}\delta_k^{2}}{L_1}\right\}, \forall x\in \mathcal{B}(\widehat{x}^{k},\delta_k).
	\]
\end{itemize}
\noindent Then,
\[
\Phi(\widehat{x}^k + s^k) - \Phi(\widehat{x}^k) \leq -C_1 \|\nabla\Phi(\widehat{x}^k)\|\,\delta_k
\]
holds with probability at least $\alpha$, where $C_1 \coloneqq \frac{4\kappa_{dcp}\,\kappa_{ef}}{\bigl(8\kappa_{ef} + 2\kappa_{ed}\kappa_{dcp}\bigr)\max \{\delta_{\max}, 1\}}$.
\end{lem}
\begin{proof}	
	Obviously, we have
		\begin{equation*}
		\begin{aligned}
			\Phi(\widehat{x}^k + s^k) - \Phi(\widehat{x}^k) = & \underbrace{\Phi(\widehat{x}^k + s^k)-\Phi^{k}(\widehat{x}^k + s^k)}_{I_3}+\underbrace{\Phi^{k}(\widehat{x}^k + s^k) - \mathcal{L}^{k}(\widehat{x}^k+ s^k,y^{i,*}(\widehat{x}^k+ s^k))}_{I_2}\\
			&+\underbrace{\mathcal{L}^{k}(\widehat{x}^k+ s^k,y^{i,*}(\widehat{x}^k+ s^k))- \mathcal{L}^{k}(\widehat{x}^k,y^{i,*}(\widehat{x}^k))}_{I_1}\\
			&+\underbrace{\mathcal{L}^{k}(\widehat{x}^k,y^{i,*}(\widehat{x}^k))-\Phi^k(\widehat{x}^k)}_{I_2} +\underbrace{\Phi^k(\widehat{x}^k) - \Phi(\widehat{x}^k)}_{I_3}.
		\end{aligned}
	\end{equation*}
	For $I_1$, by the sufficient descent inequality \eqref{TR_algo_inex_suffdescent},
	\begin{equation}\label{TR_Lemma1_eq_6}
		\begin{aligned}
			I_1 &\leq -\kappa_{dcp} \|\nabla_{1} \mathcal{L}^{k}(\widehat{x}^{k},y^{i,*}(\widehat{x}^{k}))\| \min \left\{\delta_{k}, 1\right\}\\ &\leq -\kappa_{dcp}\frac{\delta_k}{\max \{\delta_{\max}, 1\}} \|\nabla_{1} \mathcal{L}^{k}(\widehat{x}^{k},y^{i,*}(\widehat{x}^{k}))\|.
		\end{aligned}
	\end{equation}
	For $I_2$, by $L_1$-Lipschitz continuity of $l(\cdot)$ and the setting of inexact solution $y^{i,*}(x)$,
	\begin{equation}\label{TR_Lemma1_eq_5}
		\begin{aligned}
			|\mathcal{L}^{k}(x,y^{i,*}(x)) - \Phi^{k}(x)|\leq L_1\left\|y^{i,*}(x)-y^{k,*}(x) \right\|\leq \kappa_{ef}\delta_k^2, \forall x\in\mathcal{B}(\widehat{x}^{k},\delta_k).
		\end{aligned}
	\end{equation}
	For $I_3$, if $I_k$ defined in \eqref{TR_notion_I_k} holds, $$\left |\Phi(x) -\widehat{\Phi}^{k}(x) \right| \leq \kappa_{ef}\delta_{k}^{2}.$$

	\noindent Then,
	\begin{equation*}
		\begin{aligned}
			\Phi(\widehat{x}^k + s^k) - \Phi(\widehat{x}^k) 
			&\leq 4 \kappa_{ef} \delta_k^2 - \frac{\kappa_{dcp}}{\max \{\delta_{\max}, 1\}} \|\nabla_{1} \mathcal{L}^{k}(\widehat{x}^{k},y^{i,*}(\widehat{x}^{k}))\| \delta_k.
		\end{aligned}
	\end{equation*}
	Moreover, by $\widehat{\ell}_1$-smoothness of $\mathcal{L}^{k}(\cdot)$,
	\begin{equation*}
		\begin{aligned}
			&\|\nabla\Phi^{k}(\widehat{x}^{k})\|-\|\nabla_{1} \mathcal{L}^{k}(\widehat{x}^{k},y^{i,*}(\widehat{x}^{k}))\|\\
			=&\|\nabla_{x}\mathbb{E}_{\widehat{\epsilon}^k}[l(\widehat{x}^k,y^{k,*}(\widehat{x}^k),m_k(\widehat{x}^{k},\widehat{\epsilon}^k))]\|-\|\nabla_{x}\mathbb{E}_{\widehat{\epsilon}^k}[l(\widehat{x}^k,y^{i,*}(\widehat{x}^k),m_k(\widehat{x}^{k},\widehat{\epsilon}^k))]\|\\
			\leq&\mathbb{E}_{\widehat{\epsilon}^k}\|\nabla_{x}l(\widehat{x}^k,y^{k,*}(\widehat{x}^k),m_k(\widehat{x}^{k},\widehat{\epsilon}^k))-\nabla_{x}l(\widehat{x}^k,y^{i,*}(\widehat{x}^k),m_k(\widehat{x}^{k},\widehat{\epsilon}^k))\|\\
			\leq&\widehat{\ell}_{1}\left\|y^{k,*}(\widehat{x}^k)-y^{i,*}(\widehat{x}^k) \right\|,
		\end{aligned}
	\end{equation*}
	where the first inequality follows from triangle inequality and Jensen’s inequality. If $I_k$ defined in \eqref{TR_notion_I_k} holds,
	\begin{equation*}\label{TR_Lemma1_ineq2}
		\begin{aligned}
			\|\nabla\Phi(\widehat{x}^{k})\|-\|\nabla\Phi^k(\widehat{x}^{k})\| \leq \kappa_{ed}\delta_{k}.
		\end{aligned}
	\end{equation*}
	Summarizing the above two inequalities and using the fact that $\|y^{k,*}(x)-y^{i,*}(x)\|\leq\min\{\frac{\kappa_{ed}\delta_k}{\widehat{\ell}_1},\frac{\kappa_{ef}\delta_k^{2}}{L_1}\}$, we have
	\begin{align*}
		\|\nabla_{1} \mathcal{L}^{k}(\widehat{x}^{k},y^{i,*}(\widehat{x}^{k}))\|\geq \|\nabla\Phi(\widehat{x}^{k})\|-2 \kappa_{ed}\delta_k.
	\end{align*}
	Combining the above inequality with inequality \eqref{TR_Lemma1_equ_1}, we have
	\begin{align}
		\|\nabla_{1} \mathcal{L}^{k}(\widehat{x}^{k},y^{i,*}(\widehat{x}^{k}))\|\geq& \frac{8\kappa_{ef}\max\{\delta_{\max}, 1\}}{\kappa_{dcp}} \delta_k \label{TR_Lemma1_eq_3},\\
		\|\nabla_{1} \mathcal{L}^{k}(\widehat{x}^{k},y^{i,*}(\widehat{x}^{k}))\|\geq& \frac{4\kappa_{ef}}{4\kappa_{ef}+\kappa_{ed}\kappa_{dcp}}\|\nabla\Phi(\widehat{x}^{k})\|.\label{TR_Lemma1_eq_4}
	\end{align}
	
	Then, if $I_k$ holds,
	\begin{equation*}
		\begin{aligned}
			\Phi(\widehat{x}^k + s^k) - \Phi(\widehat{x}^k) 
			&\leq - \frac{\kappa_{dcp}}{2 \max \{\delta_{\max}, 1\}} \|\nabla_{1} \mathcal{L}^{k}(\widehat{x}^{k},y^{i,*}(\widehat{x}^{k}))\| \delta_k\\
			&\leq - \frac{2\kappa_{dcp}\kappa_{ef}}{\left(4\kappa_{ef}+\kappa_{ed}\kappa_{dcp} \right)\max \{\delta_{\max}, 1\}} \|\nabla\Phi(\widehat{x}^k)\| \delta_k,
		\end{aligned}
	\end{equation*}
	where the first inequality follows from \eqref{TR_Lemma1_eq_3} and the second inequality follows from \eqref{TR_Lemma1_eq_4}.
	Proposition \ref{TR_proposition_1} implies that $I_k$ holds with probability at least $\alpha$, the above inequality also holds with probability at least $\alpha$.
	The proof is complete.
\end{proof}

\begin{lem}\label{TR_Lemma5}
Suppose that
\begin{itemize} 
	\item[\rm(a)] The number of samples for value estimates satisfies
	$$\max\{|S_{k}|, |S_{k+1/2}|\}\geq \max\left\{ \mathcal{O}\left( \delta_k^{-4} \epsilon_F^{-2} \right),  \mathcal{O}((1-\beta)^{-2}) \right\}.$$
	\item[\rm(b)] The inexact solution $y^{is,*}(x)$ for value estimate satisfies
	$\|y^{s,*}(x) - y^{is,*}(x)\| < \dfrac{1}{3L_1} \epsilon_F \delta_k^2$.
	\item[\rm(c)] $
	\epsilon_F < \tfrac{1}{2}\,\eta_1 \eta_2 \frac{\kappa_{dcp}}{\max \{ \delta_{\max}, 1 \}}.$
\end{itemize}
Then, whenever the $k$-th iteration is successful (i.e., the trial step $s^k$ is accepted), we have
\[
\Phi(\widehat{x}^{k+1}) - \Phi(\widehat{x}^k) \leq -C_2 \delta_k^2
\]
holds with probability at least $\beta$, where
$
C_2 \coloneqq \eta_1 \eta_2 \frac{\kappa_{dcp}}{\max \{ \delta_{\max}, 1 \}} - 2 \epsilon_F.
$
\end{lem}
\begin{proof}
	Note that the $k$-th iteration is successful, if $\rho_k \geq \eta_1$ and $\|\nabla_{1} \mathcal{L}^{k}(\widehat{x}^{k},y^{i,*}(\widehat{x}^{k}))\| \geq \eta_2 \delta_k$ in Step 8 of Algorithm \ref{TR:algorithmCL} holds. By the definition of $\rho_k$, 
	\begin{equation*}
		\begin{aligned}
			v_k - v_{k+1/2} &\geq \eta_1 \left( \mathcal{L}^{k}(\widehat{x}^{k},y^{i,*}(\widehat{x}^{k})) -  \mathcal{L}^{k}(\widehat{x}^{k}+s^k,y^{i,*}(\widehat{x}^{k}+s^k)) \right)\\ &\geq \eta_1 \kappa_{dcp} \|\nabla_{1} \mathcal{L}^{k}(\widehat{x}^{k},y^{i,*}(\widehat{x}^{k}))\| \min \{ \delta_k, 1 \}\\
			& \geq \eta_1 \eta_2 \frac{\kappa_{dcp}}{\max \{ \delta_{\max}, 1 \}} \delta_k^2,
		\end{aligned}
	\end{equation*}
	where the second inequality follows from the sufficient descent inequality \eqref{TR_algo_inex_suffdescent}. 
	Then, if $I_k$ defined in \eqref{TR_notion_I_k} holds, we have
	\begin{equation*}
		\begin{aligned} 
			\Phi(\widehat{x}^k) - \Phi(\widehat{x}^k + s^k) &= \Phi(\widehat{x}^k) - v_k + (v_k - v_{k+1/2}) + \left( v_{k+1/2} - \Phi(\widehat{x}^k + s^k) \right)\\
			&\geq \eta_1 \eta_2 \frac{\kappa_{dcp}}{\max \{ \delta_{\max}, 1 \}} \delta_k^2+\left(\Phi(\widehat{x}^k) - v_k\right)+\left( v_{k+1/2} - \Phi(\widehat{x}^k + s^k) \right)\\
			&\geq \left( \eta_1 \eta_2 \frac{\kappa_{dcp}}{\max \{ \delta_{\max}, 1 \}} - 2 \epsilon_F \right) \delta_k^2.
		\end{aligned}
	\end{equation*}
	Since Proposition \ref{TR_proposition_2} implies that $J_k$ holds with probability at least $\beta$, the above inequality also holds with probability at least $\beta$.
	The proof is complete.
\end{proof}
Having established these preliminary results, we now prove the almost sure convergence of Algorithm \ref{TR:algorithmCL}.
\begin{thm}\label{TR_theorem_1}
Suppose that 
{\rm(a)} Assumptions \ref{TR_ass_primal}--\ref{ass:stronglyconcave_local} hold; 
{\rm(b)} the number of samples for local linear regression $|T_k|\geq \max\{\mathcal{O}(\delta_k^{-4}(1-\alpha)^{-1}),\mathcal{O}((1-\alpha)^{-2})\}$ for some $\alpha\in (0,1)$;
{\em(c)} the number of samples for value estimates $\max\{|S_{k}|, |S_{k+1/2}|\}\geq \max\left\{ \mathcal{O}\left( \delta_k^{-4} \epsilon_F^{-2} \right),  \mathcal{O}((1-\beta)^{-2}) \right\}$  for some $\beta\in (0,1)$;
{\rm(d)} the inexact solution satisfies $ \|y^{k,*}(x)-y^{i,*}(x)\|\leq \min \left\{\frac{\kappa_{e d} \delta_{k}}{\widehat{\ell}_{1}}, \frac{\kappa_{e f} \delta_{k}^{2}}{L_{1}}\right\}, \forall x \in \mathcal{B}\left(\widehat{x}^{k}, \delta_{k}\right)$;
{\em(e)} the inexact solution satisfies
$\|y^{s,*}(x) - y^{is,*}(x)\| \leq \dfrac{1}{3L_1} \epsilon_F \delta_k^2$, $\forall x \in \mathcal{B}\left(\widehat{x}^{k}, \delta_{k}\right)$;
{\rm(f)} the step-acceptance parameters $\eta_1, \eta_2$ and the accuracy parameter $\epsilon_F$ satisfy
\begin{equation*}
\eta_{2} \geq \frac{6 \kappa_{e f}}{\kappa_{d c p}\left(1-\eta_{1}\right)}, \text { and } \epsilon_{F} \leq \min \left\{\kappa_{e f}, \frac{1}{4} \eta_{1} \eta_{2} \frac{\kappa_{d c p}}{\delta_{\max }}\right\};
\end{equation*}
{\rm(g)} the probability parameters $\alpha, \beta$ satisfy
\begin{equation*}
	\frac{\alpha \beta - \frac{1}{2}}{(1 - \alpha)(1 - \beta)} \geq \frac{C_4}{C_1}, \quad \text{and} \quad \frac{(2\alpha - 1)(2\beta - 1)}{(1 - \alpha)(1 - \beta)} \geq \frac{\nu C_3}{(1 - \nu)(\gamma - 1)}.
\end{equation*}
Then, the sequence of trust-region radii $\{\Delta_k\}$ satisfies $\sum_{k=0}^{\infty}\Delta_k^2 <\infty$ almost surely. 
Moreover,
$$\lim _{k \rightarrow \infty} \left\|\nabla\Phi(\widehat{X}^{k})\right\|=0,$$ and the random iterates $\{\widehat{X}^{k}\}$ generated by Algorithm \ref{TR:algorithmCL} converges to the stationary point almost surely.
\end{thm}

\begin{proof}
The framework of the proof is adapted from \cite[Theorem 1]{liu2022coupled}, with modifications to account for the minimax structure and the inexact solution in the algorithm. Our proof mainly includes two steps, first we show that $\sum_{k=0}^{\infty} \Delta_k^2$ is finite almost surely, based on which, we prove the iterates converge to the stationary point almost surely by contradiction.

\textbf{Step 1:}
We define the random function $$\Psi_k = \nu\Phi(\widehat{X}_k) + (1 - \nu)(\Delta_k^2 - \Delta_k)$$ with $\nu \in (0, 1)$ satisfying
\begin{equation}
\label{TR_prop6_equ_0}
\begin{aligned}
\frac{\nu}{1 - \nu} > \max \left\{ \frac{4\gamma^2}{\zeta C_1}, \frac{4\gamma^2}{C_2}, \frac{\gamma^2}{C_3} \right\},
\end{aligned}
\end{equation}
where $\gamma$ is the trust region adjustment parameter in Algorithm \ref{TR:algorithmCL}, $C_1$ and $C_2$ are defined in Lemmas \ref{TR_Lemma4} and \ref{TR_Lemma5}, respectively, $C_3 \coloneqq L_{1}(1+L_0+\kappa)$ with $\kappa=\frac{\ell_{1}(1+L_{0})}{\mu}$.
Since $\Phi(\cdot)$ is bounded from below by Assumption \ref{TR_ass_primal}, and $\Delta_k \in [0, \delta_{\max}]$, $\Psi_k$ is bounded from below for all $k$.
If there exists a constant $\tau > 0$ such that for all $k\in\mathbb{N}$
\begin{equation}
\label{prop4-2}
\begin{aligned}
\mathbb{E}[\Psi_{k+1} - \Psi_k | \mathcal{F}_{k-1}] \leq -\tau \Delta_k^2 < 0, 
\end{aligned}
\end{equation}
we can deduce $\sum_{k=0}^{\infty} \Delta_k^2$ is finite.

To prove inequality \eqref{prop4-2}, we consider two cases depending on whether $\|\nabla\Phi(\widehat{x}^k)\| \geq \zeta \delta_k$, where
\begin{equation}\label{equ_prop4_2}
	\zeta \geq 2\kappa_{ed} \max\{\delta_{\max},1\} + \max \left\{ \eta_{2}, \frac{8\kappa_{ef} \max\{\delta_{\max},1\}}{\kappa_{dcp}(1 - \eta_1)} \right\}.
\end{equation}
Within each case, there are four possible outcomes of the events \( I_k \) and \( J_k \), each occurring with a certain probability. For each outcome, we develop an upper bound for \( \phi_{k+1} - \phi_k \), where \( \phi_k \) denotes the realization of \( \Psi_k \). By combining the bounds for all four outcomes, we obtain
\[
\mathbb{E}[\Psi_{k+1} - \Psi_k \mid \mathcal{F}_{k-1}],
\]
which leads to \eqref{prop4-2} for all $k\in\mathbb{N}$. We present only the analysis of the case \(\|\nabla\Phi(\widehat{x}^k)\| \geq \zeta \delta_k\), as the complementary case follows by similar adaptation of the proof of \cite[Theorem 4.11]{Chen2018MP}.

First, we characterize \( \phi_{k+1} - \phi_k \), depending on whether the iteration is successful.
For successful iterations, \( \widehat{x}_{k+1} = \widehat{x}_k + s^k \), then, 
\begin{equation}\label{TR_equ_pro4_3}
\begin{aligned}
\phi_{k+1} - \phi_k = \nu(\Phi(\widehat{x}^{k+1}) - \Phi(\widehat{x}^k)) + (1 - \nu)(\gamma^2 - 1)\delta_k^2 + (1 - \nu)(1 - \gamma)\delta_k.
\end{aligned}
\end{equation}
For unsuccessful iterations, \( \widehat{x}_{k+1} = \widehat{x}_k \) and \( \delta_{k+1} = \delta_k / \gamma \), then, 
\begin{equation}\label{TR_prop_equ_b1}
\begin{aligned}
\phi_{k+1} - \phi_k = (1 - \nu)\left( \frac{1}{\gamma^2} - 1 \right)\delta_k^2 + (1 - \nu)\left( 1 - \frac{1}{\gamma} \right)\delta_k \coloneqq b_1.
\end{aligned}
\end{equation}

Next, we analyze \( \phi_{k+1} - \phi_k \) in detail for \( \|\nabla\Phi(\widehat{x}^k)\| \geq \zeta \delta_k \), considering the outcomes of \( I_k \) and \( J_k \).

\textbf{(i)} $I_{k}$ and $J_{k}$ are both true.
 Provided $\eta_{1} \in(0,1)$, by inequality \eqref{equ_prop4_2} we have  
$$\|\nabla\Phi(\widehat{x}^{k})\| \geq\left(2\kappa_{e d}+\frac{8 \kappa_{ef}}{\kappa_{dc p}}\right)\max\{\delta_{\max},1\}\delta_{k}.$$ Moreover, condition (d) holds. Then, we have by Lemma \ref{TR_Lemma4} that 
\begin{equation*}
\label{TR_equ_pro4_4}
\Phi\left(\widehat{x}^{k}+s^{k}\right)-\Phi\left(\widehat{x}^{k}\right) \leq-C_{1} \|\nabla\Phi(\widehat{x}^{k})\| \delta_{k},
\end{equation*}
where $C_{1} \coloneqq \frac{4 \kappa_{d c p} \kappa_{e f}}{\left(8 \kappa_{e f}+2 \kappa_{e d} \kappa_{d e p}\right) \max \left\{\delta_{\max }, 1\right\}}$. 

Given the event $I_k$ holds, we can deduce
\begin{equation*}
	\begin{aligned}
\|\nabla\Phi^{k}(\hat{x}^{k})\| &\geq \|\nabla\Phi(\hat{x}^{k})\|-\kappa_{e d} \delta_{k} \geq\left(\zeta-\kappa_{ed}\right) \delta_{k} \\
&\geq (2\max\{\delta_{max},1\}-1)\kappa_{ed}\delta_{k}+\max \left\{\eta_{2}, \frac{8 \kappa_{e f} \max\{\delta_{\max },1\}}{\kappa_{d c p}\left(1-\eta_{1}\right)}\right\} \delta_{k}.
	\end{aligned}
\end{equation*}
On the other hand, by $\widehat{\ell}_1$-smoothness of $\mathcal{L}^{k}(\cdot)$,
\begin{equation*}
	\begin{aligned}
		&\|\nabla\Phi^{k}(\hat{x}^{k})\|-\|\nabla_{x}\mathcal{L}^{k}(\widehat{x}^k,y^{i,*}(\widehat{x}^k))\|\\
		\leq& \|\nabla_{x}\mathcal{L}^{k}(\widehat{x}^k,y^{k,*}(\widehat{x}^k))-\nabla_{x}\mathcal{L}^{k}(\widehat{x}^k,y^{i,*}(\widehat{x}^k))\|\\
		\leq & \widehat{\ell}_{1} \|y^{k,*}(\widehat{x}^k)-y^{i,*}(\widehat{x}^k)\|\\
		\leq & (2\max\{\delta_{max},1\}-1)\kappa_{ed}\delta_k,
	\end{aligned}
\end{equation*}
where the last inequality follows from condition (d).
Combining the above two inequalities, we have
$$\|\nabla_{x}\mathcal{L}^{k}(\widehat{x}^k,y^{i,*}(\widehat{x}^k))\| \geq \max \left\{\eta_{2}, \frac{8 \kappa_{e f} \max\{\delta_{\max },1\}}{\kappa_{d c p}\left(1-\eta_{1}\right)}\right\} \delta_{k},$$
which implies that Lemma \ref{TR_Lemma3} holds provided $\epsilon_{F} \leq \kappa_{ef}$.
Hence, the $k$-th iteration is successful, and consequently inequality \eqref{TR_equ_pro4_3} holds. Combining with inequality \eqref{TR_equ_pro4_4}, 
\begin{equation}\label{TR_prop_equ_b2}
\phi_{k+1}-\phi_{k}\leq-\nu C_{1} \|\nabla\Phi(\widehat{x}^{k})\| \delta_{k}+(1-\nu)\left(\gamma^{2}-1\right) \delta_{k}^{2}+(1-\nu)(1-\gamma) \delta_{k} \coloneqq b_{2}.
\end{equation}

\textbf{(ii)} $I_{k}$ is true and $J_{k}$ is false. 
Given the event $I_k$ holds, by a similar analysis as (i), Lemma \ref{TR_Lemma4} holds. If the iteration is successful,  \eqref{TR_prop_equ_b2} holds; otherwise, \eqref{TR_prop_equ_b1} holds. Provided $\frac{\nu}{1-\nu} \geq \frac{4\gamma^{2}}{C_{1} \zeta}$, the right-hand side of \eqref{TR_prop_equ_b2} is strictly smaller than that of \eqref{TR_prop_equ_b1}, i.e.,
$$b_{2} - b_{1} \leq -\nu C_{1} \zeta \delta_{k}^{2} + (1 - \nu) \delta_{k}^{2}\left(\gamma^{2} - \frac{1}{\gamma^{2}}\right) + (1 - \nu) \delta_{k}\left(\frac{1}{\gamma} - \gamma\right) < 0.$$
Combining both outcomes, we obtain
$$\phi_{k+1} - \phi_{k} \leq (1 - \nu) \left(\frac{1}{\gamma^{2}} - 1\right) \delta_{k}^{2} + (1 - \nu) \left(1 - \frac{1}{\gamma}\right) \delta_{k}.$$

\textbf{(iii)} $I_{k}$ is false and $J_{k}$ is true. If the iteration is successful, we have by Lemma \ref{TR_Lemma5} that
$$\Phi\left(\widehat{x}^{k+1}\right) - \Phi\left(\widehat{x}^{k}\right) \leq -C_{2} \delta_{k}^{2}$$
with $C_{2} \coloneqq \eta_{1} \eta_{2} \frac{\kappa_{dcp}}{\max \left\{\delta_{\max}, 1\right\}} - 2 \epsilon_{F}$. Plugging the above inequality into \eqref{TR_equ_pro4_3}, we have
$$\phi_{k+1} - \phi_{k} \leq -\nu C_{2} \delta_{k}^{2} + (1 - \nu)(\gamma^{2} - 1) \delta_{k}^{2} + (1 - \nu)(1 - \gamma) \delta_{k} \coloneqq b_{3}.$$
If the iteration is unsuccessful, \eqref{TR_prop_equ_b1} holds and is strictly larger than $b_{3}$ provided $\frac{\nu}{1 - \nu} \geq \frac{4 \gamma^{2}}{C_{2}}$. 
Combining both outcomes, we obtain
$$\phi_{k+1} - \phi_{k} \leq (1 - \nu) \left(\frac{1}{\gamma^{2}} - 1\right) \delta_{k}^{2} + (1 - \nu) \left(1 - \frac{1}{\gamma}\right) \delta_{k}.$$

\textbf{(iv)} $I_{k}$ and $J_{k}$ are both false.
\begin{equation*}
\begin{aligned}
&\Phi(\widehat{x}^{k}+s^{k})-\Phi(\widehat{x}^{k})\\
=&\mathbb{E}_{\tilde{\epsilon}}\left[l(\widehat{x}^{k}+s^{k}, y^{*}\left(\widehat{x}^{k}+s^{k} \right), \psi(\widehat{x}^{k}+s^{k}) + \tilde{\epsilon})\right]-\mathbb{E}_{\tilde{\epsilon}}\left[l(\widehat{x}^{k}, y^{*}\left(\widehat{x}^{k}+s^{k}\right), \psi(\widehat{x}^{k}) + \tilde{\epsilon})\right]\\
&+\mathbb{E}_{\tilde{\epsilon}}\left[l(\widehat{x}^{k}, y^{*}\left(\widehat{x}^{k}+s^{k}\right), \psi(\widehat{x}^{k}) + \tilde{\epsilon})\right]-
\mathbb{E}_{\tilde{\epsilon}}\left[l(\widehat{x}^{k}, y^{*}\left(\widehat{x}^{k}\right), \psi(\widehat{x}^{k}) + \tilde{\epsilon})\right]\\
\leq& \mathbb{E}_{\tilde{\epsilon}}[\nabla_{1}l(\widehat{x}^{k}, y^{*}\left(\widehat{x}^{k}+s^{k}\right), \psi(\widehat{x}^{k}) + \tilde{\epsilon})]^{\top}s^{k}
+\mathbb{E}_{\tilde{\epsilon}}\left[\nabla_{3}l(\widehat{x}^{k}, y^{*}\left(\widehat{x}^{k}+s^{k}\right), \psi(\widehat{x}^{k}) + \tilde{\epsilon})\right]^{\top}
\left(\psi(\widehat{x}^{k}+s^{k})-\psi(\widehat{x}^{k})\right)\\
&+\frac{\ell_1}{2}\delta_{k}^{2}+\frac{\ell_1}{2}\left\|\psi(\widehat{x}^{k}+s^{k})-\psi(\widehat{x}^{k})\right\|^{2}
+L_{1}\left\|y^{*}\left(\widehat{x}^{k}+s^{k}\right)-y^{*}\left(\widehat{x}^{k}\right)\right\|\\
\leq& L_{1}(1+L_0+\kappa)\delta_{k}+\frac{L_0^2\ell_1+\ell_1}{2}\delta_{k}^{2},
\end{aligned}
\end{equation*}
where the first inequality follows from $\ell_{1}$-smoothness and $L_{1}$-Lipschitz continuity of $l(\cdot)$, and the second inequality follows from Cauchy–Schwarz inequality, Assumption \ref{Ass:smoothness_local}, and $\kappa$-Lipschitz continuity of $y^{*}(\cdot)$ \cite[Lemma 4.3]{lin2020gradient} with $\kappa=\frac{\ell_{1}(1+L_{0})}{\mu}$, provided $\nabla_y\mathcal{L}(\cdot,y)$ is $\ell_{1}(1+L_{0})$-Lipschitz continuous and $\mathcal{L}(\cdot)$ is $\mu$-strongly concave in $y$. 

If the iteration is successful, given $\|\nabla\Phi(\widehat{x}^{k})\|\geq \zeta\delta_{k},$ we have
\begin{equation}\label{TR_prop_equ_b3}
\phi_{k+1}-\phi_{k} \leq \nu C_{3} \delta_{k}+\nu C_{4} \|\nabla\Phi(\widehat{x}^{k})\|\delta_{k}+(1-\nu)\left(\gamma^{2}-1\right) \delta_{k}^{2}+(1-\nu)(1-\gamma) \delta_{k},
\end{equation}
where with $C_{3}\coloneqq L_{1}(1+L_0+\kappa), C_{4}\coloneqq\frac{L_0^2\ell_1+\ell_1}{2\zeta}$.
Otherwise, \eqref{TR_prop_equ_b1} holds. 
Provided $\frac{\nu}{1-\nu} \geq \frac{\gamma^{2}}{C_{3}}$, \eqref{TR_prop_equ_b1} is strictly smaller than the right-hand side of \eqref{TR_prop_equ_b3}. Then, \eqref{TR_prop_equ_b3} holds whether the iteration is successful or not.

Let $\alpha=\mathbb{P}\left(I_k \mid \mathcal{F}_{k-1}\right)$ and $\beta=\mathbb{P}\left(J_k\mid \mathcal{F}_{k-1/2}\right)$.
Taking the conditional expectation over the four possible outcomes of $\{I_k, J_k\}$, i.e., averaging the bounds of $\phi_{k+1} - \phi_k$ with probabilities $\alpha\beta$, $\alpha (1-\beta)$, $(1-\alpha)\beta$, $(1-\alpha)(1-\beta)$, we obtain
\begin{equation*}
	\begin{aligned}
		&\mathbb{E}[\Psi_{k+1} - \Psi_k \mid \mathcal{F}_{k-1}]\\
		\leq & \alpha \beta \Big[ -\nu C_1 \|\nabla\Phi(\widehat{x}^{k})\| \Delta_k + (1 - \nu)(\gamma^2 - 1) \Delta_k^2 + (1 - \nu)(1 - \gamma) \Delta_k \Big] \\
		& + \left(\alpha(1 - \beta) + \beta(1 - \alpha)\right) \Bigg[ (1 - \nu)\left( \frac{1}{\gamma^2} - 1 \right) \Delta_k^2 + (1 - \nu)\left( 1 - \frac{1}{\gamma} \right) \Delta_k \Bigg] \\
		& + (1 - \alpha)(1 - \beta) \Big[ \nu C_3 \Delta_k + \nu C_4 \|\nabla\Phi(\widehat{x}^{k})\| \Delta_k + (1 - \nu)(\gamma^2 - 1) \Delta_k^2 + (1 - \nu)(1 - \gamma) \Delta_k \Big] \\
		= & \left( -\nu \alpha \beta C_1 + (1 - \alpha)(1 - \beta) \nu C_4 \right) \|\nabla\Phi(\widehat{x}^{k})\| \Delta_k 
		 + (1 - \alpha)(1 - \beta) \nu C_3 \Delta_k \\
		& + \underbrace{\left( \alpha \beta - \frac{1}{\gamma^2} (\alpha(1 - \beta) + (1 - \alpha)\beta) + (1 - \alpha)(1 - \beta) \right)}_{I_1} (1 - \nu)(\gamma^2 - 1) \Delta_k^2 \\
		& - \underbrace{\left( \alpha \beta - \frac{1}{\gamma} (\alpha(1 - \beta) + (1 - \alpha)\beta) + (1 - \alpha)(1 - \beta) \right)}_{I_2} (1 - \nu)(\gamma - 1) \Delta_k.
	\end{aligned}
\end{equation*}
By some calculation,
\begin{equation*}
	\begin{aligned}
	&I_{1} \leq (\alpha + (1 - \alpha))(\beta + (1 - \beta)) = 1,
\\
	&I_{2} \geq (\alpha - (1 - \alpha))(\beta - (1 - \beta)) = (2\alpha - 1)(2\beta - 1).
	\end{aligned}
\end{equation*}
Then,
\begin{equation*}
	\begin{aligned}
		\mathbb{E}[\Psi_{k+1} - \Psi_k \mid \mathcal{F}_{k-1}]
		\leq & \left( -\alpha \beta C_1 + (1 - \alpha)(1 - \beta) C_4 \right) \nu \|\nabla\Phi(\widehat{x}^{k})\| \Delta_k \\
		& + (1 - \nu)(\gamma^2 - 1) \Delta_k^2 \\
		& + \left( (1 - \alpha)(1 - \beta) \nu C_3 - (2\alpha - 1)(2\beta - 1)(1 - \nu)(\gamma - 1) \right) \Delta_k.
	\end{aligned}
\end{equation*}
Provided \(\|\nabla\Phi(\widehat{x}^k)\| \geq \zeta \Delta_k\) and the probability parameters $\alpha$ and $\beta$ satisfying
\begin{equation*}
	\frac{\alpha \beta - \frac{1}{2}}{(1 - \alpha)(1 - \beta)} \geq \frac{C_4}{C_1}, \quad \text{and} \quad \frac{(2\alpha - 1)(2\beta - 1)}{(1 - \alpha)(1 - \beta)} \geq \frac{\nu C_3}{(1 - \nu)(\gamma - 1)},
\end{equation*}
we have
\begin{equation*}
	\mathbb{E}[\Psi_{k+1} - \Psi_k \mid \mathcal{F}_{k-1}] \leq \left(-\frac{1}{2} C_1 \nu \zeta + (1 - \nu)(\gamma^2 - 1)\right) \Delta_k^2 \leq -\frac{1}{4} C_1 \nu \zeta \Delta_k^2,
\end{equation*}
where the last inequality follows from \( \frac{\nu}{1 - \nu} > \frac{4 \gamma^2}{\zeta C_1} \). A similar bound can be established for the case
\( \|\nabla\Phi(\hat{x}^k)\| \leq \zeta \delta_k \). By summing the above inequality across all iterations, we obtain that \( \sum_{k=0}^\infty \Delta_k^2 < \infty \) almost surely.

\textbf{Step 2:}
Based on the results established above, we now prove that \(\lim_{k \to \infty} \|\nabla\Phi(\widehat{X}_k)\| = 0\) almost surely. According to \cite[Theorem 4.16]{Chen2018MP}, we can derive the liminf convergence result that \(\liminf_{k \to \infty} \|\nabla\Phi(\widehat{X}^k)\| = 0\) almost surely under the assumptions stated in this theorem. We omit the proof for this liminf convergence result because it only needs to replace \(\|\nabla f(\widehat{X}_k)\|\) in the proof of \cite[Theorem 4.16 ]{Chen2018MP} with \(\|\nabla\Phi(\widehat{X}_k)\|\), where $f(\cdot)$ is the objective function of \cite{Chen2018MP}.
The remaining part of the proof, showing that \(\lim_{k \to \infty} \|\nabla\Phi(\widehat{X}_k)\| = 0\) almost surely, is similar to the proof of \cite[Theorem 4.3]{bandeira2013convergencetrustregionmethodsbased}, except that we need to make some modifications based on \(\|\nabla\Phi(\widehat{X}_k)\|\).
The detailed analysis below presents a proof by contradiction. 

Suppose that \(\lim_{k \to \infty} \|\nabla\Phi(\widehat{X}_k)\| = 0\) does not hold almost surely. Then with a positive probability, there exists \(\epsilon > 0\) such that \(\|\nabla\Phi(\widehat{X}_k)\| > 2\epsilon\) holds for infinitely many \(k\). For any \(\epsilon > 0\), define a sequence of random variables \(\{K_\epsilon\}\) consisting of the natural numbers \(k\) for which \(\|\nabla\Phi(\widehat{X}_k)\| > \epsilon\). Note that \(\sum_{k \in \{K_\epsilon\}} \Delta_k < \infty\) almost surely, which is implied by the proof of \cite[Lemma 4.17]{Chen2018MP}. We are going to show if such \(\epsilon\) exists such that \(\sum_{j \in \{K_\epsilon\}} \Delta_j\) diverges, and hence, we must have \(\lim_{k \to \infty} \|\nabla\Phi(\widehat{X}_k)\| = 0\) almost surely.

Because $\liminf_{k \to \infty}\|\nabla\Phi(\hat{X}^k)\|= 0$, there are infinitely many intervals of integers with a positive probability, such that each interval $\{W' + 1, \ldots, W''\}$ satisfies $0 < W' < W''$, $\|\nabla\Phi(\hat{X}^{W'})\| \leq \varepsilon$, $\|\nabla\Phi(\hat{X}^{W' + 1})\| > \varepsilon$, $\|\nabla\Phi(\hat{X}^{W''})\| > 2\varepsilon$, and for any integer $w \in (W', W'')$, $\varepsilon < \chi(\hat{X}^w) \leq 2\varepsilon$. Let $\{(W_r', W_r'')\}_r$ be an infinite sequence of such intervals. Let $(w_r', w_r'')$ be the realization of $(W_r', W_r'')$.
\begin{small}
\begin{equation*}
\epsilon<\left\|\nabla\Phi\left(\widehat{x}^{w_{r}^{\prime \prime}}\right)\right\|-\left\|\nabla\Phi\left(\widehat{x}^{w_{r}^{\prime}}\right)\right\|
=\sum_{k=w_{r}^{\prime}}^{w_{r}^{\prime \prime}-1}\left\|\nabla\Phi\left(\widehat{x}^{k+1}\right)\right\|
-\left\|\nabla\Phi\left(\widehat{x}^{k}\right)\right\|.
\end{equation*}
\end{small}
Furthermore,
\begin{equation*}
\begin{aligned}
&\left\|\nabla\Phi\left(\widehat{x}^{k+1}\right)\right\|-\left\|\nabla\Phi\left(\widehat{x}^{k}\right)\right\|\\
\leq&\left\|\nabla\Phi\left(\widehat{x}^{k+1}\right)-\nabla\Phi\left(\widehat{x}^{k}\right)\right\|\\
=&\left\|\nabla_{x}\mathbb{E}_{\tilde{\epsilon}}\left[l(\widehat{x}^{k+1}, y^{*}(\widehat{x}^{k+1}), \psi(\widehat{x}^{k+1}) + \tilde{\epsilon})\right]
-\nabla_{x}\mathbb{E}_{\tilde{\epsilon}}\left[l(\widehat{x}^{k}, y^{*}(\widehat{x}^{k}), \psi(\widehat{x}^{k}) + \tilde{\epsilon})\right]\right\|\\
\leq&\left\|\nabla_{x}\mathbb{E}_{\tilde{\epsilon}}\left[l(\widehat{x}^{k+1}, y^{*}(\widehat{x}^{k+1}), \psi(\widehat{x}^{k+1}) + \tilde{\epsilon})\right]
-\nabla_{x}\mathbb{E}_{\tilde{\epsilon}}\left[l(\widehat{x}^{k}, y^{*}(\widehat{x}^{k+1}), \psi(\widehat{x}^{k}) + \tilde{\epsilon})\right]\right\|\\
&+\left\|\nabla_{x}\mathbb{E}_{\tilde{\epsilon}}\left[l(\widehat{x}^{k}, y^{*}(\widehat{x}^{k+1}), \psi(\widehat{x}^{k}) + \tilde{\epsilon})\right]-\nabla_{x}\mathbb{E}_{\tilde{\epsilon}}\left[l(\widehat{x}^{k}, y^{*}(\widehat{x}^{k}), \psi(\widehat{x}^{k}) + \tilde{\epsilon})\right] \right\|\\
\leq&\ell_1\left(1+L_0\right)\left\|\widehat{x}^{k+1}-\widehat{x}^{k} \right\|+\ell_1\|y^{*}(\widehat{x}^{k+1})-y^{*}(\widehat{x}^{k})\|\\
\leq& L_2\left(1+L_0+\kappa\right)\left\|\widehat{x}^{k+1}-\widehat{x}^{k} \right\|,
\end{aligned}
\end{equation*}
where the third inequality follows from $\ell_1$-smoothness of $l(\cdot)$ and $L_0$-Lipschitz continuity of $\psi(\cdot)$, and the last inequality follows from $\kappa$-Lipschitz continuity of $y^{*}(\cdot)$ with $\kappa=\frac{\ell_{1}(1+L_0)}{\mu}$, provided $\nabla_y\mathcal{L}(\cdot,y)$ is $\ell_{1}(1+L_0)$-Lipschitz and $\mathcal{L}(x,\cdot)$ is $\mu$-strongly concave.
Therefore, we derive that for any $r$,
\[
\ell_1\left(1+L_0+\kappa\right)\sum_{k=W_r'}^{W_r''-1} \|\widehat{x}^{k+1}-\widehat{x}^k\| \geq \epsilon,
\]
which yields that \(\sum_{k \in [K_\epsilon]} \Delta_k = \infty\) almost surely and thus contradicts the initial assumption. Hence, we have \(\lim_{k \to \infty} \|\nabla\Phi(\widehat{X}_k)\| = 0\) almost surely. Furthermore, by Lemma 11.1.2 in Conn et al. \cite{Conn2000Trustregionmethods}, we have that the sequence produced by Algorithm \ref{TR:algorithmCL} converges to a stationary point almost surely. The proof is complete.
\end{proof}

Theorem \ref{TR_theorem_1} shows that the random iterates $\{\widehat{X}^{k}\}$ generated by Algorithm \ref{TR:algorithmCL} converge to the stationary point of SMDD \eqref{eq:f_minimax_regression} almost surely, if the trust-region model and the value estimates used in the algorithm are sufficiently accurate approximates with high probability. 
While the almost sure convergence analysis of Algorithm~\ref{TR:algorithmCL} seems like a natural extension of  \cite[Theorem~1]{liu2022coupled}, our analysis required careful handling of both the inexact step acceptance criterion and the inexact solution of the inner maximization problem, which are nontrivial challenges. 

\section{Experiments}\label{section4}
To evaluate the performance of TR, we conduct experiments on a simple synthetic example and a real-world application. In both experiments, we compare TR with the adaptive stochastic gradient descent ascent algorithm (ASGDA)~\cite{Gao2025Adaptive} and the stochastic primal-dual method  (SPD)~\cite{wood2023stochastic}. 

\subsection{A synthetic nonconvex--strongly concave minimax problem}
Consider the following stochastic minimax problem with decision-dependent distribution
\begin{equation}\label{TR_eq_exp1}
	\underset{x\in\mathbb{R}}{\operatorname{min}}\max_{y\in [-125, 125]}\mathcal{L}(x,y)\coloneqq  x^{2}-2(x+y)\mathbb{E}\left[\widetilde{\omega}\right]-y^{2},
\end{equation}
where $\widetilde{\omega}=x^3+\widetilde{\epsilon}$ and  $\widetilde{\epsilon}\sim\mathcal{N}(0,1)$. Obviously, $\mathcal{L}(\cdot)$ is nonconvex in $x$ and strongly concave in $y$, and the corresponding primal function 
\begin{equation*}
	\Phi(x)\coloneqq \max_{y\in[-125,125]}\{x^{2}-2(x+y)x^3-y^{2}\}=\left\{\begin{array}{cc}
		x^{2} - 2x^4 + 250x^3 - 15625, & x>5, \\
		x^{2} - 2x^4 + x^6, & x \in\left[-5, 5\right], \\
		x^{2} - 2x^4 - 250x^3 - 15625, & x<-5.
	\end{array}\right.
\end{equation*}
By some calculation, we have $x^{\star}=\{0, 1, -1\}$ are stationary points of SMDD \eqref{TR_eq_exp1}.

We evaluate the convergence behavior of the algorithms in terms of the norm of gradient of the primal function versus the number of iterations. 
For ASGDA, the stepsize $\eta_x=10^{-3}, \eta_y=10^{-1}$, and for SPD, the constant stepsize $\eta=10^{-3}$ and the dynamic stepsize $\eta_{k}=\frac{1}{1000+10k}$, and the batchsize $M=500$. Moreover, the initial point $(x^0, y^0)$ is randomly generated around $(10,10)$ within a radius of $0.5$.
\begin{figure}[H]
	\centering
	\begin{minipage}[t]{0.32\textwidth}
		\centering
		\subfigure[TR]{	
			\includegraphics[scale=0.38]{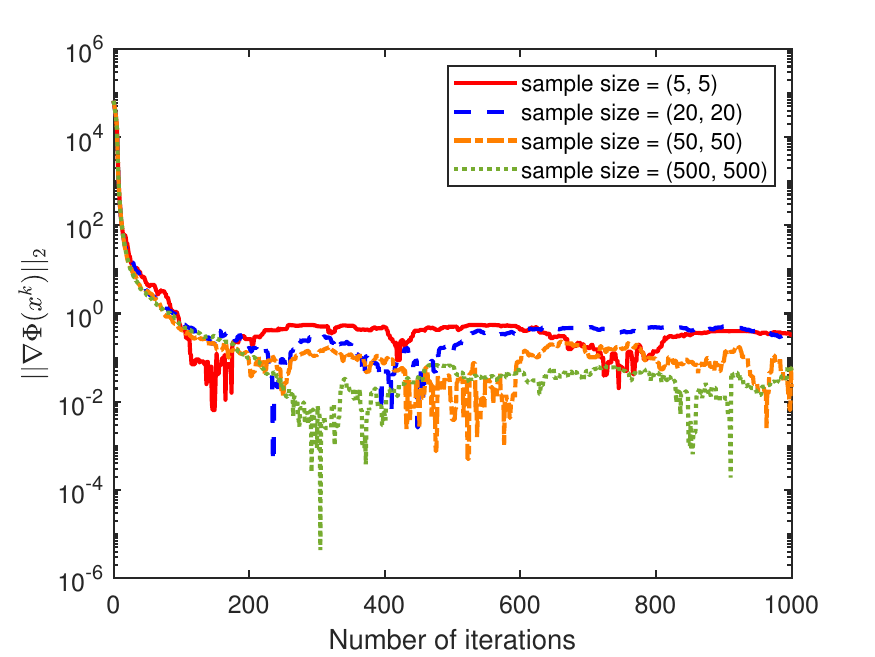}
		}
	\end{minipage}
	\begin{minipage}[t]{0.32\textwidth}
	\centering
	\subfigure[ASGDA]{		
		\includegraphics[scale=0.38]{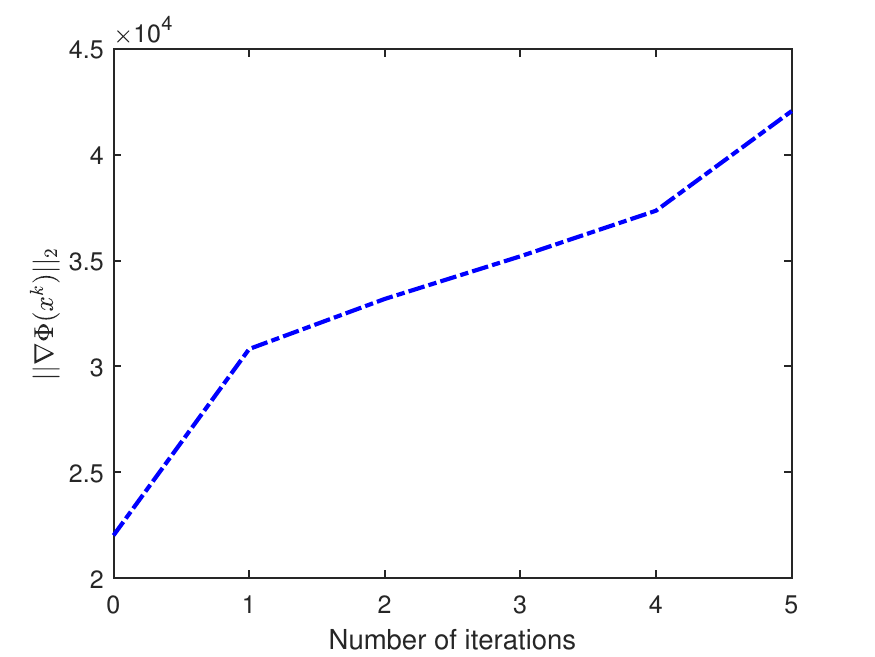}
	}
	\end{minipage}
	\begin{minipage}[t]{0.32\textwidth}
		\centering
		\subfigure[SPD]{		
			\includegraphics[scale=0.38]{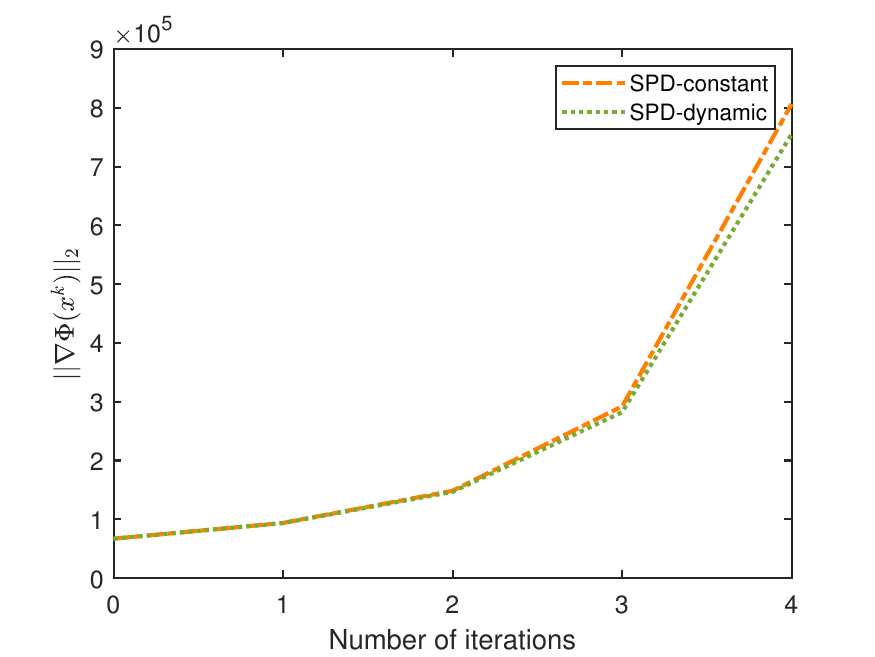}
		}
	\end{minipage}
	\caption{Performance on synthetic example}
	\label{TR_figure1}
\end{figure}

We report the performance of the three algorithms 
in Figures~\ref{TR_figure1}, where Figure~\ref{TR_figure1} (a) depicts the performance of TR with different sample sizes for local linear regression and function value estimation, Figure~\ref{TR_figure1} (b) depicts the performance of ASGDA and Figure~\ref{TR_figure1} (c) depicts the performance of SPD with constant and dynamic stepsizes. We can observe from Figure~\ref{TR_figure1} (a) that the norm of gradient tends to zero as the number of iterations increases, and the convergence accuracy gets better with larger sample size. 
From Figures~\ref{TR_figure1} (b) and (c), we can observe that ASGDA and SPD with both constant and dynamic stepsizes diverge to infinity within a finite number of iterations. 
The underlying reason may be that the global parametric model used in ASGDA can not fit the distribution map well in a large region, and SPD is designed for finding the performative equilibrium point of SMDD with strongly convex-strongly concave  objective function.

\subsection{Distributionally robust optimization problem}
Consider the following distributionally robust optimization problem~\cite{Gao2025Adaptive}
\begin{equation}\label{euq:experimet_real}
	\min _{x \in \mathbb{R}^{n}} \max _{y \in Y \subset \mathbb{R}^{N}}\mathcal{L}(x, y)\coloneqq \frac{1}{N} \sum_{i=1}^{N} y_{i}\ell(x;a_{i},b_{i})+f(x)-g(y),
\end{equation}
with $$\ell(x;a_{i},b_{i})=\log \left(1+\exp \left(-b_{i} (a_{i})^{\top} x\right)\right),$$
$$f(x)=\lambda_{1} \sum_{i=1}^{n} \frac{\alpha x_{i}^{2}}{1+\alpha x_{i}^{2}},\,\,
g(y)=\frac{1}{2} \lambda_{2}\left\|N y-\mathbf{1}\right\|^{2},$$
$\ell(\cdot)$ is the logistic loss function,
$f(\cdot)$ is a nonconvex regularizer, $g(\cdot)$ is a distributionally robust regularizer,
$Y \coloneqq\left\{y \in \mathbb{R}_{+}^{N}: \mathbf{1}^{\top} y=1\right\}$ is a simplex with $\mathbf{1}$ denotes the vector with all entries equal to one, $\mathcal{L}(\cdot)$ is the average loss function over the entire data set.
In problem \eqref{euq:experimet_real}, $x$ represents parameter of the classifier, $N$ is the total number of training samples, $a_{i}\in\mathbb{R}^{n}$ is the feature of the $i$-th sample, $b_{i}\in\{-1,1\}$ is the corresponding label.
\begin{figure}[H]
	\centering
	\begin{minipage}[t]{0.37\textwidth}
		\centering
		\subfigure[]{		
			\includegraphics[scale=0.42]{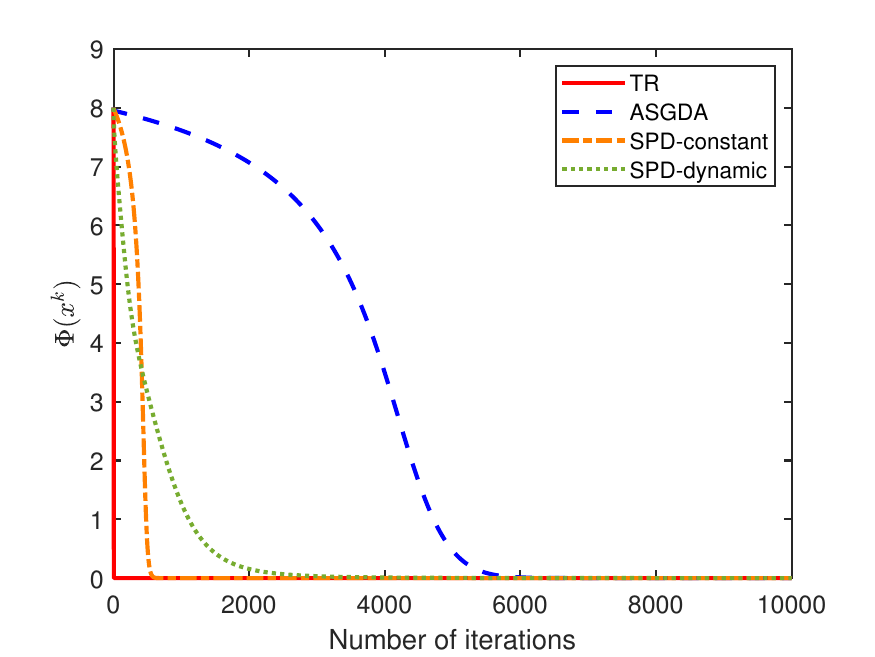}
		}
	\end{minipage}
	\begin{minipage}[t]{0.37\textwidth}
		\centering
		\subfigure[]{
			\includegraphics[scale=0.42]{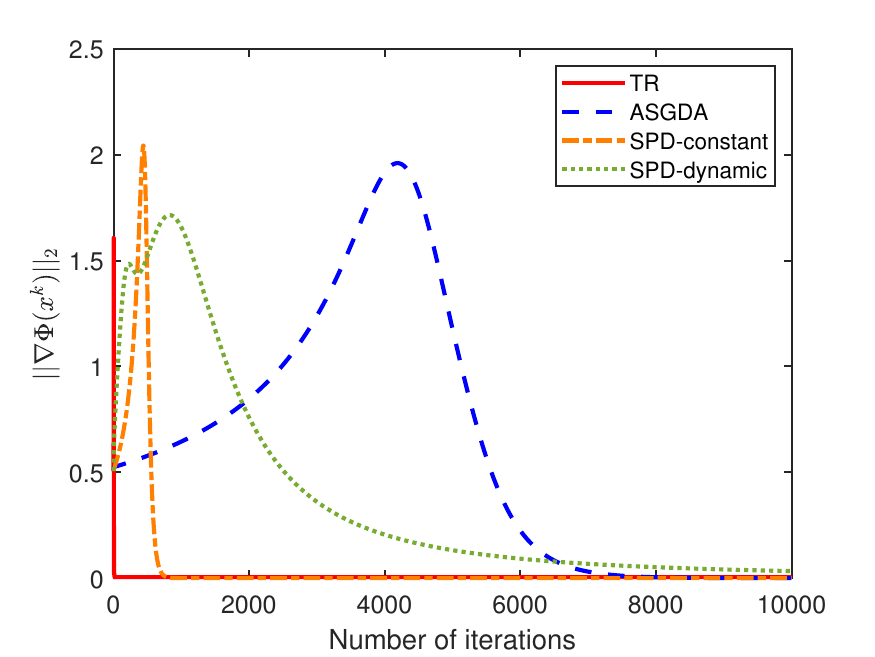}
		}
	\end{minipage}
	\caption{Performance on distributionally robust optimization}
	\label{TR_figure2}
\end{figure}
To run the simulations, we adopt the Kaggle credit scoring data set \cite{kaggle2012givemesomecredit} for loan approval in bank as the base data set to generate the data $S=\{(a_i,b_i)\}_{i=1}^N$ that depends on the decision variable $x$. Given the base data set $S_0=\{\left(a_i^0,b_i^0\right)\}_{i=1}^N$, $a_i=a_i^0+V\sin(x), b_i=b_i^0$, where $V\in\mathbb{R}^{n\times n}$ is a diagonal matrix whose diagonal entries equal to $5$.	
We set the parameters of problem (\ref{euq:experimet_real}) as $\lambda_1=1, \lambda_2=\frac{10}{N^2}, \alpha=1$.
For TR, the sample size for local linear regression and function value estimation $(N_k, M_k)=(300,100)$. For ASGDA, the stepsize $\eta_x=10^{-3}, \eta_y=10^{-1}$, and for SPD, the constant stepsize $\eta=10^{-2}$ and the dynamic stepsize $\eta_{t}=\frac{1}{10+t}$, and the batchsizes $M=200$.

We report the convergence behavior of the algorithms in Figure \ref{TR_figure2}, where Figure \ref{TR_figure2} (a) and (b) record the performance of the value of the primal function and the norm of its gradient versus the number of iterations, respectively.
We can observe from Figure \ref{TR_figure2} that the value of the primal function and the norm of its gradient corresponding to TR, ASGDA and SPD with constant and dynamic stepsizes tend to zero as the number of iterations increases. According to Figure \ref{TR_figure2}, it seems that TR algorithm needs fewer iterations to converge to the stationary point compared to ASGDA and SPD. Indeed, TR requires more samples than ASGDA and SPD, as it runs local linear regression and solves a maximization problem inexactly at each iteration.

\noindent\textbf{Acknowledgment.} The research is supported by the NSFC \#12471283 and Fundamental Research Funds for the Central Universities DUT24LK001.

\end{document}